\documentclass[
	english
]{scrartcl}
\usepackage[T1]{fontenc}
\usepackage[utf8]{inputenc}

\usepackage{amsmath,amsfonts,amsthm,amssymb}
\usepackage[colorlinks,citecolor=blue]{hyperref}
\usepackage{doi,natbib}
\usepackage{cancel}

\usepackage{changes}

\usepackage[figure]{hypcap}
\usepackage{graphicx}
\usepackage{subcaption} 

\usepackage{preamble}

\usepackage{marginnote}

\usepackage[capitalise,nameinlink]{cleveref}
\allowdisplaybreaks

\newcommand\norm[2]{\left\lVert#1\right\rVert_{#2}}

\newcommand\N{\mathbb{N}}
\newcommand\R{\mathbb{R}}

\newcommand{\cl}{\operatorname{cl}}

\newcommand{\spa}{\operatorname{span}}
\newcommand{\conv}{\operatorname{conv}}
\newcommand{\cone}{\operatorname{cone}}
\newcommand{\dist}{\operatorname{dist}}
\newcommand{\clconv}{\operatorname{\overline{\conv}}}

\DeclareMathOperator*{\argmin}{\operatorname{argmin}}

\DeclareMathAlphabet{\mathpzc}{OT1}{pzc}{m}{it}
\newcommand\oo{\mathpzc{o}}

\newtheorem{theorem}{Theorem}[section]
\newtheorem{lemma}[theorem]{Lemma}
\newtheorem{proposition}[theorem]{Proposition}

\newtheorem{corollary}[theorem]{Corollary}
\newtheorem{remark}[theorem]{Remark}
\newtheorem{definition}[theorem]{Definition}
\newtheorem{example}[theorem]{Example}

\crefname{figure}{Figure}{Figures}

\begin{document}

\title{On the linear independence constraint qualification in disjunctive programming}
%\subtitle{}
\author{%
	Patrick Mehlitz%
	\footnote{%
		Brandenburgische Technische Universität Cottbus--Senftenberg,
		Institute of Mathematics,
		03046 Cottbus,
		Germany,
		\email{mehlitz@b-tu.de},
		\url{https://www.b-tu.de/fg-optimale-steuerung/team/dr-patrick-mehlitz},
		ORCID: 0000-0002-9355-850X%
		}
	}

% \date{\today}
\publishers{}
\maketitle

\begin{abstract}
	 Mathematical programs with disjunctive constraints (MPDCs for short) cover several different
	 problem classes from nonlinear optimization including complementarity-, 
	 vanishing-, cardinality-, and switching-constrained optimization problems.
	 In this paper, we introduce an abstract but reasonable version of the prominent linear independence
	 constraint qualification which applies to MPDCs. Afterwards, we derive first- and
	 second-order optimality conditions for MPDCs under validity of this constraint
	 qualification based on so-called strongly stationary points. 
	 Finally, we apply our findings to some popular classes of disjunctive
	 programs and compare the obtained results to those ones available in the literature.
	 Particularly, new second-order optimality conditions for mathematical programs with
	 switching constraints are by-products of our approach.
\end{abstract}


\begin{msc}	
	90C30, 90C33
\end{msc}

\section{Introduction}\label{sec:introduction}

In this paper, so-called \emph{mathematical programs with disjunctive constraints} (MPDCs) are studied. 
These are optimization problems of the form
\begin{equation}\label{eq:MPDC}\tag{MPDC}
	\begin{aligned}
		f(x)&\,\rightarrow\,\min\\
		F(x)&\,\in\,D
	\end{aligned}
\end{equation}
where $f\colon\R^n\to\R$ as well as $F\colon\R^n\to\R^m$ are twice continuously differentiable and
$D:=\bigcup_{i=1}^rD_i$ is the finite union of given polyhedral sets $D_1,\ldots,D_r\subset\R^m$.
Recall that a set is called polyhedral whenever it can be represented as the intersection of
finitely many half spaces.
We use $X:=\{x\in\R^n\,|\,F(x)\in D\}$ in order to denote the feasible set of \eqref{eq:MPDC}.
Clearly, choosing $r:=1$ and $D:=\R^{p}_-\times\{0^q\}$ with $p,q\in\mathbb N$ and $m:=p+q$,
any standard nonlinear program is a disjunctive program, see \cref{ex:standard_nonlinear_programming}. 
However, it is well known that the model \eqref{eq:MPDC} covers 
\emph{mathematical programs with complementarity constraints} (MPCCs), see \cite{LuoPangRalph1996}, 
\emph{mathematical programs with vanishing constraints} (MPVCs), see \cite{AchtzigerKanzow2008},
\emph{cardinality-constrained mathematical programs} (CCMPs), see \cite{PanXiuFan2017}, 
and \emph{mathematical programs with switching constraints} (MPSCs), see \cite{Mehlitz2018},
as well.
Note that all these problem classes which frequently arise from the mathematical modeling of real-world
applications suffer from an inherent lack of regularity. That is why huge effort has been put into the
derivation of problem-tailored stationarity notions and constraint qualifications. However, it is clear
that any theoretical result which can be derived for the generalized model \eqref{eq:MPDC} 
has a corresponding counterpart for MPCCs, MPVCs, CCMPs, and MPSCs. This observation justifies the theoretical
investigation of \eqref{eq:MPDC}.
First ideas on how to study disjunctive structures in nonlinear optimization are presented in \cite{Scholtes2004}.
Stationarity notions and constraint qualifications for \eqref{eq:MPDC} can be found in
\citep{BenkoGfrerer2018,FlegelKanzowOutrata2007,Gfrerer2014}. 
Particularly, second-order necessary and sufficient optimality conditions for disjunctive programs are
derived in \cite{Gfrerer2014} with the aid of the celebrated directional limiting calculus, 
see \cite{Gfrerer2013} as well. Checking \cite[Theorems~3.3, 3.17]{Gfrerer2014}, one can observe that in
contrast to classical second-order optimality conditions, the appearing set of multipliers depends on the
choice of the particular critical direction.

In this paper, we are going to state an MPDC-tailored version of the \emph{linear independence constraint qualification}
(LICQ)
and study its inherent properties. Furthermore, we derive second-order necessary optimality conditions for \eqref{eq:MPDC}
under validity of this constraint qualification
in a completely elementary way using second-order tangent sets. Thus, our approach is related to techniques
which were used in \cite{BonnansShapiro2000,ChristofWachsmuth2018,Penot1998,RockafellarWets1998} to derive 
second-order conditions for mathematical programs.  
On the other hand, we present a result which shows the isolatedness of strongly stationary
points of \eqref{eq:MPDC} where the problem-tailored version of LICQ and a suitable second-order
sufficient condition hold. This generalizes some corresponding results for MPCCs, see
\cite{GuoLinYe2013}, and CCMPs, see \cite[Corollary~3.3]{BucherSchwartz2018}. 
Afterwards, we apply our findings to several instances of disjunctive programming.
In particular, new second-order optimality conditions for MPSCs will be derived.

The remaining parts of this paper are structured as follows: In \cref{sec:preliminaries}, we comment on
the notation used in this manuscript and introduce all the necessary tools from variational analysis which
are exploited later. Furthermore, some preliminary results are provided. \cref{sec:MPDC_LICQ} is dedicated
to the derivation of an MPDC-tailored version of the linear independence constraint qualification.
Some consequences of the validity of this regularity condition are presented. 
Second-order optimality conditions for \eqref{eq:MPDC} are the topic of interest in \cref{sec:second_order_conditions}.
First, it will be shown that a second-order necessary optimality condition holds at the local minimizers
of \eqref{eq:MPDC} where our new constraint qualification is valid. Afterwards, a second-order sufficient
optimality condition for \eqref{eq:MPDC} will be derived. Subsequently, we show that this condition together with
the problem-tailored version of LICQ implies that the underlying strongly stationary point of interest 
is in a certain sense locally isolated.
In \cref{sec:application}, we apply the derived theory to MPCCs, MPVCs, as well as CCMPs and compare our findings to
available results from the literature, see 
\cite{ScheelScholtes2000,HoheiselKanzow2007,BucherSchwartz2018}. 
Furthermore, we obtain new second-order optimality conditions for MPSCs.
Some final remarks close the paper in \cref{sec:conclusions}.

\section{Preliminaries}\label{sec:preliminaries}

\subsection{Basic notation}\label{sec:basic_notation}

Throughout this paper, $x\cdot y$ is used to denote the common Euclidean inner product of two
vectors $x,y\in\R^n$. We equip $\R^n$ with the Euclidean norm $\norm{\cdot}{2}$. 
The zero vector in $\R^n$ will be denoted by $0^n$ while $0$ is used to represent 
the scalar zero. For $\varepsilon>0$
and some $\bar x\in\R^n$, 
$\mathbb B^\varepsilon(\bar x):=\{x\in\R^n\,|\,\norm{x-\bar x}{2}\leq\varepsilon\}$
denotes the closed $\varepsilon$-ball around $\bar x$. Similarly, 
$\mathbb U^\varepsilon(\bar x):=\{x\in\R^n\,|\,\norm{x-\bar x}{2}<\varepsilon\}$
represents the open $\varepsilon$-ball around $\bar x$.
Frequently, we will make use of the sets $\R_+:=\{t\in\R\,|\,t\geq 0\}$
and $\R_-:=\{t\in\R\,|\,t\leq 0\}$. For a given nonempty set $A\subset\R^n$, we exploit
$\cl A$, $\cone A$, $\conv A$, $\clconv A$, and $\spa A$ in order to represent the closure of $A$,
the conic hull of $A$, 
the convex hull of $A$, the closed convex hull of $A$, and the span of $A$
(i.e.\ the smallest subspace of $\R^n$ comprising $A$), respectively. 
We use $\dist(x,A):=\inf\{\norm{z-x}{2}\,|\,z\in A\}$ to represent the distance
of $x\in\R^n$ to $A$. 
Finally, the Cartesian product $A\times B$ of two sets $A\subset\R^n$ and $B\subset \R^m$
will be interpreted as a subset of $\R^{n+m}$.

Recall that a set-valued mapping $\Psi\colon\R^n\rightrightarrows\R^m$, i.e.\ a mapping which
assigns to each $x\in\R^n$ a (possibly empty) set $\Psi(x)\subset\R^m$, is called metrically
subregular at some point $(\bar x,\bar y)\in\{(x,y)\in\R^n\times\R^m\,|\,y\in\Psi(x)\}$ if
there are constants $\kappa>0$ and $\varepsilon>0$ such that
\[
	\forall x\in\mathbb U^\varepsilon(\bar x)\colon\quad
		\dist(x,\Psi^{-1}(\bar y))\leq\kappa\dist(\bar y,\Psi(x))
\]
holds true. Here, $\Psi^{-1}\colon\R^m\rightrightarrows\R^n$ denotes the inverse set-valued
mapping associated with $\Psi$ which is defined by $\Psi^{-1}(y):=\{x\in\R^n\,|\,y\in\Psi(x)\}$ for
all $y\in\R^m$. It is easily seen that $\Psi$ is metrically subregular at $(\bar x,\bar y)$
if and only if $\Psi^{-1}$ possesses the so-called calmness property at $(\bar y,\bar x)$,
see e.g.\ \cite{HenrionOutrata2005}.

For a twice continuously differentiable mapping $P\colon\R^n\to\R^m$,
$\nabla P(x)\in\R^{m\times n}$ denotes its Jacobian at $x\in\R^n$.
In the particular case $m=1$, the gradient $\nabla P(x)$ will be interpreted as
a column vector.
Furthermore, we set
\[
	\forall x\in\R^n\,\forall d,h\in\R^n\colon\quad
	\nabla^2P(x)[d,h]:=\begin{pmatrix}
		d^\top\nabla^2P_1(x)h\\\vdots\\d^\top\nabla^2P_m(x)h
	\end{pmatrix}
\]
where $P_1,\ldots,P_m\colon\R^n\to\R$ are the component mappings
associated with $P$ while the matrices $\nabla^2P_1(x),\ldots,\nabla^2P_m(x)$ are their respective
Hessians at $x\in\R^n$.
\subsection{Variational analysis}\label{sec:variational_analysis}

Here, we introduce the notions of variational analysis which are necessary in order to carry out
our later considerations. For terminology and notation, we mainly follow 
\cite{AubinFrankowska2009,BonnansShapiro2000,RockafellarWets1998}.

\subsubsection{Polars and annihilators}\label{sec:polars}
For a nonempty set $A\subset\R^n$, the polar cone and the annihilator of $A$ are given as stated below:
\[
	A^\circ:=\{y\in\R^n\,|\,\forall x\in A\colon\,x\cdot y\leq 0\},\qquad
	A^\perp:=\{y\in\R^n\,|\,\forall x\in A\colon\,x\cdot y=0\}.
\]
Obviously, $A^\circ$ is a closed, convex cone and one has $A^\perp=A^\circ\cap(-A)^\circ$, i.e.\ $A^\perp$
is a subspace of $\R^n$. For any two sets $A,B\subset\R^n$, one easily obtains
$(A\cup B)^\circ=A^\circ\cap B^\circ$ as well as
$(A\cup B)^\perp = A^\perp\cap B^\perp$. 
For a cone $C\subset\R^n$, the so-called bipolar theorem, see \cite[Corollary~6.21]{RockafellarWets1998},
shows $C^{\circ\circ}=\clconv C$. Furthermore, the polarization rule $(C_1+C_2)^\circ=C_1^\circ\cap C_2^\circ$ follows
for any two cones $C_1,C_2\subset\R^n$, see \cite[Section~2.1.4]{BonnansShapiro2000}. 
If $C_1,C_2$ are additionally, closed and convex, we have $(C_1\cap C_2)^\circ=\cl(C_1^\circ+C_2^\circ)$.
Particularly, for subspaces $L_1,L_2\subset\R^n$, $(L_1\cap L_2)^\perp=L_1^\perp+L_2^\perp$ is valid since
each subspace of $\R^n$ is closed.
Supposing that $K\subset\R^n$ is a closed, convex cone, one obtains
$K^{\circ\perp}=K\cap(-K)$ from the bipolar theorem, 
i.e.\ $K^{\circ\perp}$ coincides with the so-called lineality space of $K$ which
is the largest subspace contained in $K$. Additionally,
\begin{align*}
	K^{\perp\perp}
		&=(K^\circ\cap(-K)^\circ)^\perp=(K^\circ\cap(-K)^\circ)^\circ\\
		&=\cl(K^{\circ\circ}+(-K)^{\circ\circ})	=\cl(K-K)=K-K=\spa K
\end{align*}
follows from the calculation rules provided above.

\subsubsection{Tangent and normal cones}

Let $A\subset\R^n$ be closed and fix an arbitrary point $\bar x\in A$. The closed cones
\begin{align*}
	\mathcal T_A(\bar x)&:=\left\{
							d\in\R^n\,\middle|\,
								\begin{aligned}
									&\exists\{t_k\}_{k\in\N}\subset\R_+\,\exists\{d_k\}_{k\in\N}\subset\R^n\colon\\
									&\qquad	t_k\downarrow 0,\,d_k\to d,\,\bar x+t_kd_k\in A\,\forall k\in\mathbb N
								\end{aligned}
							\right\},\\
	\mathcal T^\flat_A(\bar x)&:=\left\{
							d\in\R^n\,\middle|\,
								\begin{aligned}
									&\forall\{t_k\}_{k\in\N}\subset\R_+,\,t_k\downarrow 0\\
									&\qquad	\exists\{d_k\}_{k\in\N}\subset\R^n\colon\,d_k\to d,\,\bar x+t_kd_k\in A\,\forall k\in\mathbb N
								\end{aligned}
							\right\},\\
	\mathcal T_A^\textup{c}(\bar x)&:=\left\{
							d\in\R^n\,\middle|\,
								\begin{aligned}
									&\forall \{t_k\}_{k\in\N}\subset\R_+,\,t_k\downarrow 0\,\forall\{x_k\}_{k\in\N}\subset A,\,x_k\to\bar x\\
									&\qquad \,\exists\{d_k\}_{k\in\N}\subset\R^n\colon
									d_k\to d,\,x_k+t_kd_k\in A\,\forall k\in\mathbb N
								\end{aligned}
							\right\}
\end{align*}
are called the tangent (or Bouligand) cone, the inner (or adjacent) tangent cone, and the Clarke 
tangent cone to $A$ at $\bar x$, respectively. By definition, we always have the inclusions
$\mathcal T_A^\textup c(\bar x)\subset \mathcal T_A^\flat(\bar x)\subset \mathcal T_A(\bar x)$, and all these
cones coincide whenever $A$ is convex. If we have $\mathcal T_A^\flat(\bar x)= \mathcal T_A(\bar x)$, then $A$
is said to be derivable at $\bar x$. We say that $A$ is derivable if it is derivable at each of its points.
The cone $\widehat{\mathcal N}_A(\bar x):=\mathcal T_A(\bar x)^\circ$ is referred to as
Fr\'{e}chet (or regular) normal cone. By definition, it is closed and convex.
Furthermore, we exploit the limiting (or Mordukhovich) normal cone to
$A$ at $\bar x$ which is given by
\[
	\mathcal N_A(\bar x)=\left\{\eta\in\R^n\,\middle|\,
		\begin{aligned}
			&\exists \{x_k\}_{k\in\N}\subset A\,\exists \{\eta_k\}_{k\in\N}\subset\R^n\colon\\
			&\qquad	\,x_k\to\bar x,\,\eta_k\to\eta,\,\eta_k\in\widehat{\mathcal N}_A(x_k)\,\forall k\in\mathbb N
		\end{aligned}
	\right\}.
\]
Using the notion of the Painlevé-Kuratowski-limit, see e.g.\ \cite[Section~4.B]{RockafellarWets1998}, we have
\[
	\mathcal N_A(\bar x)=\limsup\limits_{x\to\bar x,\,x\in A}\widehat{\mathcal N}_A(x).
\]
Clearly, $\widehat{\mathcal N}_A(\bar x)\subset\mathcal N_A(\bar x)$ holds and both cones coincide whenever $A$ is convex.
In case where $A$ is a closed, convex cone, we obtain $\widehat{\mathcal N}_A(\bar x)=A^\circ\cap\{\bar x\}^\perp$.
For formal completeness, we set $\mathcal T_A(x)=\mathcal T^\flat_A(x)=\mathcal T^\textup{c}_A(x):=\varnothing$
and $\widehat{\mathcal N}_A(x)=\mathcal N_A(x):=\varnothing$ for each $x\notin A$.
\begin{lemma}\label{lem:Frechet_normals_to_polyhedral_set}
	Let $Q\subset\R^n$ be a polyhedral set and fix $\bar x\in Q$.
	Then, there is some $\varepsilon>0$ such that we have
	\[
		\forall x\in Q\cap\mathbb U^\varepsilon(\bar x)\colon\quad
		\widehat{\mathcal N}_Q(x)=\widehat{\mathcal N}_Q(\bar x)\cap\{x-\bar x\}^\perp.
	\]
\end{lemma}
\begin{proof}
	Since $Q$ is polyhedral, \cite[Exercise~6.47]{RockafellarWets1998} yields the existence of
	$\varepsilon>0$ such that 
	$Q\cap\mathbb U^\varepsilon(\bar x)=(\{\bar x\}+\mathcal T_Q(\bar x))\cap\mathbb U^\varepsilon(\bar x)$
	is valid. 
	Now, fix an arbitrary point $x\in Q\cap\mathbb U^\varepsilon(\bar x)$. Noting that $x$ is an interior point
	of $\mathbb U^\varepsilon(\bar x)$, we obtain
	\begin{align*}
		\widehat{\mathcal N}_Q(x)
		&=
		\widehat{\mathcal N}_{Q\cap\mathbb U^\varepsilon(\bar x)}(x)
		=
		\widehat{\mathcal N}_{(\{\bar x\}+\mathcal T_Q(\bar x))\cap\mathbb U^\varepsilon(\bar x)}(x)
		=
		\widehat{\mathcal N}_{\{\bar x\}+\mathcal T_Q(\bar x)}(x)\\
		&=
		\widehat{\mathcal N}_{\mathcal T_Q(\bar x)}(x-\bar x)
		=
		\mathcal T_Q(\bar x)^\circ\cap\{x-\bar x\}^\perp
		=
		\widehat{\mathcal N}_Q(\bar x)\cap\{x-\bar x\}^\perp
	\end{align*}
	from the fact that $\mathcal T_Q(\bar x)$ is a closed, convex cone.
	This completes the proof.
\end{proof}
\begin{lemma}\label{lem:cones_to_disjunctive_structures}
	Let $A:=\bigcup_{i=1}^rA_i$ be the finite union of closed sets $A_1,\ldots,A_r\subset\R^n$, choose $\bar x\in A$, and set
	$I(\bar x):=\{i\in\{1,\ldots,r\}\,|\,\bar x\in A_i\}$.
	Then, one has
	\[
		\begin{aligned}
		\mathcal T_A(\bar x)&=\bigcup\limits_{i\in I(\bar x)}\mathcal T_{A_i}(\bar x),&\qquad
		\mathcal T_A^\flat(\bar x)&\supset\bigcup\limits_{i\in I(\bar x)}\mathcal T_{A_i}^\flat(\bar x),&\qquad
		\mathcal T_A^\textup c(\bar x)&\supset\bigcap\limits_{i\in I(\bar x)}\mathcal T_{A_i}^\textup{c}(\bar x),&\\
		\widehat{\mathcal N}_A(\bar x)&=\bigcap\limits_{i\in I(\bar x)}\widehat{\mathcal N}_{A_i}(\bar x),&\qquad
		\mathcal N_A(\bar x)&\subset\bigcup\limits_{i\in I(\bar x)}\mathcal N_{A_i}(\bar x).& &&
		\end{aligned}
	\]
	If, additionally, the sets $A_1,\ldots,A_r$ are convex, then $A$ is derivable.
	Furthermore, we particularly have
	\[
		\mathcal N_A(\bar x)\subset\bigcup\limits_{i\in I(\bar x)}\widehat{\mathcal N}_{A_i}(\bar x).
	\]
\end{lemma}
\begin{proof}
	The formulas for the tangent and the inner tangent cone can be found in 
	\cite[Tables~4.1 and 4.2]{AubinFrankowska2009}.
	Furthermore, the identity for the Fr\'{e}chet normal cone follows from
	the formula for the tangent cone by polarization. 
	The inclusion for the Clarke tangent cone follows by definition of this cone
	while observing that $A$ is the union of only finitely many sets.
	For the proof of the inclusion involving the limiting normal cone, observe that due to the closedness of all the sets
	$A_1,\ldots,A_r$, there is some ball $\mathbb B^\varepsilon(\bar x)$
	such that $I(x)\subset I(\bar x)$ holds for all $x\in A\cap\mathbb B^\varepsilon(\bar x)$. This yields
	\begin{align*}
		\mathcal N_A(\bar x)
			&=\limsup\limits_{x\to\bar x,\,x\in A}\widehat{\mathcal N}_A(x)
			 =\limsup\limits_{x\to\bar x,\,x\in A}\bigcap\limits_{i\in I(x)}\widehat{\mathcal N}_{A_i}(x)\\
			&\subset\limsup\limits_{x\to\bar x,\,x\in A}\bigcup\limits_{i\in I(x)}\widehat{\mathcal N}_{A_i}(x)
			 \subset\limsup\limits_{x\to\bar x,\,x\in A}\bigcup\limits_{i\in I(\bar x)}\widehat{\mathcal N}_{A_i}(x)\\
			&=\bigcup\limits_{i\in I(\bar x)}\limsup\limits_{x\to\bar x,\,x\in A_i}\widehat{\mathcal N}_{A_i}(x)
			 =\bigcup\limits_{i\in I(\bar x)}\mathcal N_{A_i}(\bar x).
	\end{align*}
	
	Now, assume that $A_1,\ldots,A_r$ are convex. From above, we obtain
	\[
		\mathcal T_A(\bar x)
		=
		\bigcup\limits_{i\in I(\bar x)}\mathcal T_{A_i}(\bar x)
		=
		\bigcup\limits_{i\in I(\bar x)}\mathcal T_{A_i}^\flat(\bar x)
		\subset
		\mathcal T_A^\flat(\bar x)
		\subset
		\mathcal T_A(\bar x)
	\]
	for each $\bar x\in A$ since each convex set is derivable. 
	Thus, $A$ is derivable. Taking the above upper estimate for the
	limiting normal cone in mind, the final formula of the lemma
	follows by convexity of $A_1,\ldots,A_r$.
\end{proof}

\begin{lemma}\label{lem:stability_property}
	Let $S:=\bigcup_{i=1}^rS_i$ be the finite union of polyhedral sets $S_1,\ldots,S_r\subset\R^n$.
	Fix a sequence $\{x_k\}_{k\in\N}\subset S$ converging to $\bar x\in\R^n$. 
	For each $k\in\N$, let $\lambda_k\in\widehat{\mathcal N}_S(x_k)$ be chosen such that
	$\lambda_k\to\bar\lambda$ holds true for some $\bar\lambda\in\R^n$.	
	Then, we have 
	$\lambda_k\cdot(x_k-\bar x)=\bar\lambda\cdot(x_k-\bar x)=0$ 
	for all sufficiently large $k\in\N$.

\end{lemma}
\begin{proof}
	Let us set $I(x):=\{i\in\{1,\ldots,r\}\,|\,x\in S_i\}$ for each $x\in S$, 
	see \cref{lem:cones_to_disjunctive_structures}.
	Due to the closedness of $S$, $\bar x\in S$ is valid.
	Exploiting the closedness of $S_1,\ldots,S_r$ as well as the convergence $x_k\to\bar x$, 
	the inclusion $I(x_k)\subset I(\bar x)$ needs to be valid for all sufficiently large $k\in\N$. 
	Particularly, for each large enough $k_0\in\N$, there is an index $i(k_0)\in I(\bar x)$ 
	such that $x_{k_0}\in S_{i(k_0)}$
	holds. Now, we can exploit \cref{lem:Frechet_normals_to_polyhedral_set} 
	in order to see the relation
	$\lambda_{k_0}\in\widehat{\mathcal N}_{S_{i(k_0)}}(x_{k_0})=\widehat{\mathcal N}_{S_{i(k_0)}}(\bar x)\cap\{x_{k_0}-\bar x\}^\perp$
	for large enough $k_0\in\N$, i.e.\ $\lambda_{k_0}\cdot(x_{k_0}-\bar x)=0$ follows.
	
	Noting that $S$ is a finite union, for sufficiently large $k_0\in\N$, there is a subsequence 
	$\{x_{k_l}\}_{l\in\N}$ of $\{x_k\}_{k\in\N}$ with $k_l\geq k_0$ for all $l\in\N$ and some
	index $i(k_0)\in I(x_{k_0})$ such that
	$i(k_0)\in I(x_{k_l})\subset I(\bar x)$ is valid for all $l\in\N$. 
	For large enough $l\in\N$, we particularly obtain
	$\lambda_{k_l}\in\widehat{\mathcal N}_{S_{i(k_0)}}(x_{k_l})$
	from \cref{lem:cones_to_disjunctive_structures}.
	Noting that there exist	only finitely many different Fr\'{e}chet normal cones to a polyhedral set, 
	we find a subsequence $\{x_{k_{l_\nu}}\}_{\nu\in\N}$ of $\{x_{k_l}\}_{l\in\N}$ and a polyhedral cone $K\subset\R^n$ such that
	$K=\widehat{\mathcal N}_{S_{i(k_0)}}(x_{k_{l_\nu}})$ holds for all $\nu\in\N$, and for large enough $k_0\in\N$, we can even
	guarantee $K=\widehat{\mathcal N}_{S_{i(k_0)}}(x_{k_0})$. Particularly,	$\lambda_{k_{l_\nu}}\in K$ follows for all $\nu\in\N$. 
	Noting that $\lambda_{k_{l_\nu}}\to\bar \lambda$ holds as $\nu\to\infty$,
	we have $\bar\lambda\in K$ by closedness of $K$. Finally, we observe that
	$K=\widehat{\mathcal N}_{S_{i(k_0)}}(x_{k_0})=\widehat{\mathcal N}_{S_{i_{k_0}}}(\bar x)\cap\{x_{k_0}-\bar x\}^\perp$
	holds due to \cref{lem:Frechet_normals_to_polyhedral_set}. 
	This shows $\bar\lambda\cdot(x_{k_0}-\bar x)=0$ for large enough $k_0\in\N$.
\end{proof}

\subsubsection{Second-order tangent sets}

For the consideration of second-order optimality conditions, we exploit so-called second-order tangent sets.
Therefore, let $A\subset\R^n$ be a closed set and fix $\bar x\in A$ as well as $d\in \mathcal T_A(\bar x)$.
The closed sets
\begin{align*}
	\mathcal T_A^2(\bar x;d)&:=	
		\left\{
			h\in\R^n\,\middle|\,
				\begin{aligned}
					&\exists\{t_k\}_{k\in\N}\subset\R_+\,\exists\{h_k\}_{k\in\N}\subset\R^n\colon\\
					&\qquad t_k\downarrow 0,\,h_k\to h,\,\bar x+t_kd+\tfrac12t_k^2h_k\in A\,\forall k\in\N
				\end{aligned}
		\right\},\\
	\mathcal T_A^{\flat,2}(\bar x;d)&:=	
		\left\{
			h\in\R^n\,\middle|\,
				\begin{aligned}
					&\forall\{t_k\}_{k\in\N}\subset\R_+,\,t_k\downarrow 0\\
					&\qquad \exists\{h_k\}_{k\in\N}\subset\R^n\colon\, h_k\to h,\,\bar x+t_kd+\tfrac12t_k^2h_k\in A\,\forall k\in\N
				\end{aligned}
		\right\}
\end{align*}
are called outer (Bouligand) and inner (adjacent) second-order tangent set to $A$ at $\bar x$ in direction $d$,
see e.g.\ \cite[Definition~3.28]{BonnansShapiro2000}. Note that these sets are not conic in general.
For $\tilde d\notin\mathcal T_A(\bar x)$, we set 
$\mathcal T^2_A(\bar x;\tilde d)=\mathcal T^{\flat,2}_A(\bar x,\tilde d):=\varnothing$
for formal completeness.
Clearly, we always have $\mathcal T_A^{\flat,2}(\bar x;d)\subset\mathcal T_A^2(\bar x;d)$. 
If equality holds, then $A$ is called parabolically derivable at $\bar x$ in direction $d$. We say that $A$ is
parabolically derivable if it is parabolically derivable at each point $x\in A$ in each direction $d\in \mathcal T_A(x)$.
Note that even convex sets are not parabolically derivable in general. However, it follows from
\cite[Proposition~3.34]{BonnansShapiro2000} that each polyhedral set $Q\subset\R^n$ is parabolically derivable and
it holds
\begin{equation}\label{eq:polyhedral_set_parabolically_derivable}
	\forall x\in Q\,\forall d\in\mathcal T_Q(x)\colon\,
	\quad
	\mathcal T^2_Q(x;d)=\mathcal T^{\flat,2}_Q(x;d)=\mathcal T_{\mathcal T_Q(x)}(d),
\end{equation}
see \cite[Proposition~13.12]{RockafellarWets1998} as well.
From \cref{lem:cones_to_disjunctive_structures}, we know that the union of finitely many polyhedral
sets is derivable. 
In the subsequent lemma, we extend this result to parabolic derivability.
\begin{lemma}\label{lem:disjunctive_sets_parabolically_derivable}
	Let $S_1,\ldots,S_r\subset\R^n$ be polyhedral sets and define $S:=\bigcup\nolimits_{i=1}^rS_i$
	Then, for each $x\in S$ and $d\in\mathcal T_S(x)$, we have
	\[
		\mathcal T^2_S(x;d)=\mathcal T^{\flat,2}_S(x;d)=\mathcal T_{\mathcal T_S(x)}(d)
	\]
	and
	\[
		\mathcal T^2_S(x;d)+\bigcap\limits_{i\in I(x)}\mathcal T_{S_i}(x)^{\circ\perp}
			\subset \mathcal T^2_S(x;d)
	\]
	where we used $I(x):=\{i\in\{1,\ldots,r\}\,|\,x\in S_i\}$.
	In particular, $S$ is parabolically derivable.
\end{lemma}
\begin{proof}
	Fix $x\in S$ and $d\in\mathcal T_S(x)$. 
	We exploit the calculus rules from \cite[Proposition~3.37]{BonnansShapiro2000} and
	the fact that polyhedral sets are parabolically derivable
	in order to obtain
	\[
		\mathcal T_S^2(x;d)
		=
		\bigcup\limits_{i\in I(x)}\mathcal T_{S_i}^2(x;d)
		=
		\bigcup\limits_{i\in I(x)}\mathcal T^{\flat,2}_{S_i}(x;d)
		\subset
		\mathcal T^{\flat,2}_S(x;d)
		\subset
		\mathcal T^2_S(x;d).
	\]
	This already shows the parabolic derivability of $S$ at $x$ in direction $d$.
	
	Next, we use formula \eqref{eq:polyhedral_set_parabolically_derivable} 
	as well as \cref{lem:cones_to_disjunctive_structures} in order to see
	\[
		\mathcal T_S^2(x;d)
		=
		\bigcup\limits_{i\in I(x)}\mathcal T_{S_i}^2(x;d)
		=
		\bigcup\limits_{i\in I(x)}\mathcal T_{\mathcal T_{S_i}(x)}(d)
		=
		\mathcal T_{\bigcup\nolimits_{i\in I(x)}\mathcal T_{S_i}(x)}(d)
		=
		\mathcal T_{\mathcal T_S(x)}(d).
	\]
	
	In order to prove correctness of the last formula, we first invoke 
	\cite[Proposition~13.12]{RockafellarWets1998} in order to see that
	\[
		\mathcal T^2_{S_i}(x;d)+\mathcal T_{S_i}(x)
			\subset \mathcal T^2_{S_i}(x;d)
	\]
	holds true for all $i\in I(x)$ since $S_i$ is a polyhedron.
	This leads to
	\begin{align*}
		\mathcal T^2_S(x;d)&+\bigcap\limits_{i\in I(x)}\mathcal T_{S_i}(x)^{\circ\perp}
		= \left(\bigcup\limits_{j\in I(x)}\mathcal T^2_{S_j}(x;d)\right)
					+\bigcap\limits_{i\in I(x)}\mathcal T_{S_i}(x)^{\circ\perp}\\
		&=	\bigcup\limits_{j\in I(x)}\left(\mathcal T^2_{S_j}(x;d)
					+\bigcap\limits_{i\in I(x)}\mathcal T_{S_i}(x)^{\circ\perp}\right)
		\subset\bigcup\limits_{j\in I(x)}\left(\mathcal T^2_{S_j}(x;d)
					+\mathcal T_{S_j}(x)^{\circ\perp}\right)\\
		&\subset\bigcup\limits_{j\in I(x)}\left(\mathcal T^2_{S_j}(x;d)
					+\mathcal T_{S_j}(x)\right)
		\subset\bigcup\limits_{j\in I(x)}\mathcal T^2_{S_j}(x;d)
		 =\mathcal T^2_S(x;d)
	\end{align*}
	and completes the proof.
\end{proof}

\subsubsection{Inverse images}

Next, we present some preliminary results on the variational geometry associated with preimages of
closed sets under smooth transformations.
Therefore, let $P\colon\R^n\to\R^m$ be a twice continuously differentiable mapping and let $\Omega\subset\R^m$
be a closed set such that $Y:=\{x\in\R^n\,|\,P(x)\in\Omega\}$ is nonempty.
For our subsequent considerations, we fix $\bar x\in Y$.

Let us first discuss variational approximations of tangents to $Y$ at $\bar x$.
We call
\begin{equation}\label{eq:linearization_cone}
	\mathcal L_Y(\bar x):=\{d\in\R^n\,|\,\nabla P(\bar x)d\in\mathcal T_\Omega(P(\bar x))\}
\end{equation}
the linearization cone to $Y$ at $\bar x$. One always has the inclusion 
$\mathcal T_Y(\bar x)\subset\mathcal L_Y(\bar x)$, see \cite[Theorem~6.31]{RockafellarWets1998}, while
equality holds if the so-called feasibility mapping $\R^n\ni x\mapsto \{P(x)\}-\Omega\subset\R^m$
is metrically subregular at $(\bar x,0^m)$, see \cite[Proposition~1]{HenrionOutrata2005}.
The latter condition has been named \emph{metric subregularity constraint qualification} (MSCQ)
in \cite[Definition~3.2]{GfrererMordukhovich2015}. 
As it is mentioned in \cite{HenrionOutrata2005}, the validity of the so-called
\emph{no nonzero abnormal multiplier constraint qualification} (NNAMCQ) given by
\[
	0^n=\nabla P(\bar x)^\top\lambda,\,\lambda\in\mathcal N_\Omega(P(\bar x))
	\,\Longrightarrow\,\lambda=0^m
\]
is sufficient for MSCQ to hold. 
In the literature, NNAMCQ is sometimes called \emph{generalized Mangasarian-Fromovitz constraint qualification} (GMFCQ)
since it reduces to the classical MFCQ condition in the context of standard nonlinear programming.
Below, we show that MSCQ can be used in order to obtain a precise
characterization of the outer second-order tangent set associated with $Y$.
For the proof, we follow ideas from \cite[Proposition~13.13]{RockafellarWets1998}.
\begin{lemma}\label{lem:outer_second_order_tangents_preimage}
	Let $\bar x\in Y$ be arbitrarily chosen.
	Then, for each $d\in\mathcal T_Y(\bar x)$, we have
	\[
		\mathcal T^2_Y(\bar x;d)
		\subset
		\left\{
			h\in\R^n\,\middle|\,
			\nabla P(\bar x)h+\nabla^2P(\bar x)[d,d]\in\mathcal T^2_\Omega(P(\bar x);\nabla P(\bar x)d)
		\right\}.
	\]
	If MSCQ is valid at $\bar x$, then equality holds.
\end{lemma}
\begin{proof}
	We start proving the inclusion $\subset$ which is supposed to be valid in general.
	Fix $h\in \mathcal T^2_Y(\bar x;d)$ arbitrarily.
	Then, we find sequences $\{h_k\}_{k\in\N}\subset\R^n$ and $\{t_k\}_{k\in\N}\subset\R_+$
	such that $h_k\to h$, $t_k\downarrow 0$, as well as $P(\bar x+t_kd+\tfrac12t_k^2h_k)\in \Omega$
	for all $k\in\N$ hold true. For each $k\in\N$, we now define
	\[
		r_k:=\frac{P(\bar x+t_kd+\tfrac12t_k^2h_k)-P(\bar x)-t_k\nabla P(\bar x)d}{\tfrac12t_k^2}.
	\]
	Then, we have $P(\bar x)+t_k\nabla P(\bar x)d+\tfrac12t_k^2r_k=P(\bar x+t_kd+\tfrac12t_k^2h_k)\in \Omega$
	for all $k\in\N$, i.e.\ supposing that $\{r_k\}_{k\in\N}$ converges, its limit belongs
	to $\mathcal T^2_\Omega(P(\bar x);\nabla P(\bar x)d)$. On the other hand, we have
	\[
		r_k=\frac{P(\bar x+t_k(d+\tfrac12t_kh_k))-P(\bar x)-t_k\nabla P(\bar x)(d+\tfrac12t_kh_k)}{\tfrac12t_k^2}
			+\nabla P(\bar x)h_k
	\]
	and this sum converges to $\nabla ^2P(\bar x)[d,d]+\nabla P(\bar x)h$, see e.g.\ \cite[Example~13.8]{RockafellarWets1998}.
	Thus, we have $\nabla ^2P(\bar x)[d,d]+\nabla P(\bar x)h\in\mathcal T^2_\Omega(P(\bar x);\nabla P(\bar x)d)$.\\
	Now, we assume that MSCQ holds at $\bar x$ and show validity of the converse inclusion $\supset$.
	Fix $h\in\R^n$ satisfying 
	$\nabla P(\bar x)h+\nabla^2P(\bar x)[d,d]\in \mathcal T^2_\Omega(P(\bar x);\nabla P(\bar x)d)$.
	Then, 	we find sequences $\{r_k\}_{k\in\N}\subset\R^m$ and $\{t_k\}_{k\in\N}\subset\R_+$ such that 
	$r_k\to\nabla P(\bar x)h+\nabla^2P(\bar x)[d,d]$, $t_k\downarrow 0$, 
	and $P(\bar x)+t_k\nabla P(\bar x)d+\tfrac12t_k^2r_k\in \Omega$ for all $k\in\N$.
	Noting that $Y$ is closed, let us fix $x_k\in\argmin\{\norm{\bar x+t_kd+\tfrac12t_k^2h-x}{2}\,|\,x\in Y\}$ for all $k\in\N$.
	Since the feasibility map $\R^n\ni x\mapsto \{P(x)\}-\Omega\subset\R^m$ is metrically subregular at $(\bar x,0^m)$, 
	we find constants $\kappa>0$ and $\varepsilon>0$ such that
	\[
		\forall x\in\mathbb U^\varepsilon(\bar x)\colon\quad
			\dist(x,Y)\leq\kappa\dist(P(x),\Omega).
	\]
	We obtain
	\begin{align*}
			&\norm{\frac{x_k-\bar x-t_kd}{\tfrac12t_k^2}-h}{2}
			 =\frac2{t_k^2}\dist(\bar x+t_kd+\tfrac12t_k^2h,Y)
			 \leq\frac{2\kappa}{t_k^2}\dist(P(\bar x+t_kd+\tfrac12t_k^2h),\Omega)\\
			%&\qquad\leq\frac{2\kappa}{t_k^2}\norm{P(\bar x+t_kd+\tfrac12t_k^2h)-P(\bar x)-t_k\nabla P(\bar x)d-\tfrac12t_k^2r_k}{2}\\
			&\qquad\leq\kappa\norm{\frac{P(\bar x+t_kd+\tfrac12t_k^2h)-P(\bar x)-t_k\nabla P(\bar x)d}{\tfrac12t_k^2}-r_k}{2}\\
			&\qquad =\kappa\norm{\frac{P(\bar x+t_k(d+\tfrac12t_kh))-P(\bar x)-t_k\nabla P(\bar x)(d+\tfrac12t_kh)}{\tfrac12t_k^2}
					+\nabla P(\bar x)h-r_k}{2}\\
	\end{align*}
	for sufficiently large $k\in\N$ and the last term converges to $0$ as $k\to\infty$, see \cite[Example~13.8]{RockafellarWets1998}.
	Thus, we have shown $h\in\mathcal T_Y^2(\bar x;d)$.
\end{proof}

Let us now focus on the variational description of Fr\'{e}chet normals to $Y$ at $\bar x$.
Similarly as above, we generally have
\begin{equation}\label{eq:Frechet_normal_cone_inverse_image}
	\widehat{\mathcal N}_Y(\bar x)\supset \nabla P(\bar x)^\top\widehat{\mathcal N}_\Omega(P(\bar x))
\end{equation}
while the converse inclusion can only be guaranteed postulating additional assumptions. 
The following result is associated with this issue and taken from
\cite[Theorem~4]{GfrererOutrata2016}.
\begin{proposition}\label{prop:subspaceCQ}
	Let $\bar x\in Y$ be a point where MSCQ holds.
	Suppose that there exists a subspace $L\subset\R^m$ satisfying 
	$\mathcal T_\Omega(P(\bar x))+L\subset\mathcal T_\Omega(P(\bar x))$ and
	\[
		\nabla P(\bar x)\R^n+L=\R^m.
	\]
	Then, equality holds in \eqref{eq:Frechet_normal_cone_inverse_image}.	
\end{proposition}

In the context of certain instances of disjunctive programming, there exist weaker conditions than those
ones postulated in \cref{prop:subspaceCQ} which ensure equality in \eqref{eq:Frechet_normal_cone_inverse_image},
see \cite{BenkoGfrerer2017}.

\section{An MPDC-tailored version of the linear independence constraint qualification}\label{sec:MPDC_LICQ}

We start this section by defining the constraint qualification of our interest. 
Recall that $X\subset\R^n$ denotes the feasible set of \eqref{eq:MPDC}.
\begin{definition}\label{def:MPDC-LICQ}
	Let $\bar x\in X$ be an arbitrary feasible point of \eqref{eq:MPDC}. 
	Then, the \emph{linear independence constraint qualification} (MPDC-LICQ) is said to hold
	at $\bar x$ if the following condition is valid:
	\[
		0^n=\nabla F(\bar x)^\top\lambda,\,
		\lambda\in\sum\limits_{i\in I(\bar x)}\spa\widehat{\mathcal N}_{D_i}(F(\bar x))\,
		\Longrightarrow\,
		\lambda=0^m.
	\]
	Here, we used $I(\bar x):=\{i\in\{1,\ldots,r\}\,|\,F(\bar x)\in D_i\}$.
\end{definition}

We first note that MPDC-LICQ holds at $\bar x\in X$ whenever the matrix $\nabla F(\bar x)$
possesses full row rank $m$, i.e.\ if the gradients $\nabla F_i(\bar x),\ldots,\nabla F_m(\bar x)$
of the component mappings $F_1,\ldots,F_m\colon\R^n\to\R$ associated with $F$ are linearly independent.
In the example below, it will be demonstrated that MPDC-LICQ reduces to the well-known LICQ
whenever standard nonlinear programs are under consideration.
\begin{example}\label{ex:standard_nonlinear_programming}
	For continuously differentiable functions $g_1,\ldots,g_p,h_1,\ldots,h_q\colon\R^n\to\R$, 
	we consider the standard nonlinear program
	\begin{equation}
		\label{eq:NLP}\tag{NLP}
			\begin{aligned}
				f(x)&\,\to\,\min&&&\\
				g_j(x)&\,\leq\,0&\quad&j=1,\ldots,p&\\
				h_j(x)&\,=\,0&\quad&j=1,\ldots,q.&
			\end{aligned}
	\end{equation}
	In order to transfer it to a program of type \eqref{eq:MPDC}, we choose $r:=1$, 
	set $D:=\R^p_-\times\{0^q\}$,
	and define $F\colon\R^n\to\R^{p+q}$ by means of
	\[
		\forall x\in\R^n\colon\quad
			F(x):=\begin{bmatrix}
				g(x)^\top&h(x)^\top
			\end{bmatrix}^\top .
	\]
	Here, the mappings $g\colon\R^n\to\R^p$ and $h\colon\R^n\to\R^q$ possess the
	component mappings $g_1,\ldots,g_p$ and $h_1,\ldots,h_q$, respectively.
	
	Fix a feasible point $\bar x\in X$ of \eqref{eq:NLP} and define
	$I^g(\bar x):=\{j\in\{1,\ldots,p\}\,|\,g_j(\bar x)=0\}$.
	Using the calculus rules for the tangent and Fr\'{e}chet normal cone to 
	Cartesian products of (convex) sets, see \cite[Proposition~6.41]{RockafellarWets1998},
	we have
	\[
		\mathcal T_{D}(F(\bar x))
			=\mathcal T_{\R^p_-}(g(\bar x))\times\mathcal T_{\{0^q\}}(h(\bar x)),\qquad
		\widehat{\mathcal N}_{D}(F(\bar x))
			=\widehat{\mathcal N}_{\R^p_-}(g(\bar x))\times\widehat{\mathcal N}_{\{0^q\}}(h(\bar x)).
	\]
	Straightforward calculations lead to the formulas
	\[
	\begin{aligned}
	\mathcal T_{\R^p_-}(g(\bar x))
		&=\{d\in\R^p\,|\,\forall j\in I^g(\bar x)\colon\,d_j\leq 0\},&
	\quad\mathcal T_{\{0^q\}}(h(\bar x))
		&=\{0^q\},\\
	\widehat{\mathcal N}_{\R^p_-}(g(\bar x))
		&=\{\lambda\in\R^p_+\,|\,\forall j\notin I^g(\bar x)\colon\,\lambda_j=0\},&
	\qquad
	\widehat{\mathcal N}_{\{0^q\}}(h(\bar x))
		&=\R^q,
	\end{aligned}
	\]
	i.e.\ we have
	\[
		\spa\widehat{\mathcal N}_D(F(\bar x))
			=\{\lambda\in\R^p\,|\,\forall j\notin I^g(\bar x)\colon\,\lambda_j=0\}\times\R^q.
	\]
	Thus, MPDC-LICQ from \cref{def:MPDC-LICQ} takes the following form for \eqref{eq:NLP}:
	\[
		\left.
			\begin{aligned}
				&0^n=\nabla g(\bar x)^\top\lambda+\nabla h(\bar x)^\top\rho,\\
				&\forall j\notin I^g(\bar x)\colon\,\lambda_j=0
			\end{aligned}
		\right\}
		\,\Longrightarrow\,
		\lambda=0^p,\,\rho=0^q.
	\]
	This is equivalent to the linear independence of the vectors from
	\[
		\{\nabla g_j(\bar x)\,|\,j\in I^g(\bar x)\}\cup\{\nabla h_j(\bar x)\,|\,j\in\{1,\ldots,q\}\}
	\]
	which is precisely the definition of the standard linear independence constraint
	qualification from nonlinear programming.
\end{example}
In \cref{sec:application}, we will show that in the particular instances of MPCCs,
MPVCs, CCMPs, and MPSCs, MPDC-LICQ coincides with the well-known respective problem-tailored
version of LICQ.

We provide an equivalent primal characterization of MPDC-LICQ in the subsequent lemma.
\begin{lemma}\label{lem:primal_MPDC_LICQ}
	Fix $\bar x\in X$ arbitrarily. Then, MPDC-LICQ is valid at $\bar x$ if and only if the subsequent
	condition is satisfies:
	\[
		\nabla F(\bar x)\R^n+\bigcap\limits_{i\in I(\bar x)}\mathcal T_{D_i}(F(\bar x))^{\circ\perp}=\R^m.
	\]
\end{lemma}
\begin{proof}
	First, we note that for any matrix $A\in\R^{m\times n}$ and any subspace $L\subset\R^m$, the equivalence
	\[
		A\R^n+L=\R^m\,\Longleftrightarrow\,\{\lambda\in L^\perp\,|\,A^\top \lambda=0^n\}=\{0^m\}
	\]
	follows from the polarization rules provided in \cref{sec:variational_analysis}. Thus, the statement
	of the lemma follows from setting $A:=\nabla F(\bar x)$ as well as 
	$L:=\bigcap_{i\in I(\bar x)}\mathcal T_{D_i}(F(\bar x))^{\circ\perp}$ and
	observing that
	\[
		\left(\bigcap_{i\in I(\bar x)}\mathcal T_{D_i}(F(\bar x))^{\circ\perp}\right)^\perp
		=\sum\limits_{i\in I(\bar x)}\widehat{\mathcal N}_{D_i}(F(\bar x))^{\perp\perp}
		=\sum\limits_{i\in I(\bar x)}\spa\widehat{\mathcal N}_{D_i}(F(\bar x))
	\]
	holds true.
\end{proof}
\begin{remark}\label{rem:nondegeneracy}
	Fix $\bar x\in X$ arbitrarily. Due to
	\begin{equation}\label{eq:subspace_candidates}
		\mathcal T_D(F(\bar x))^{\circ\perp}
		=
		\left(\bigcap\limits_{i\in I(\bar x)}\mathcal T_{D_i}(F(\bar x))^\circ\right)^\perp
		\supset
		\sum\limits_{i\in I(\bar x)}\mathcal T_{D_i}(F(\bar x))^{\circ\perp}
		\supset
		\bigcap\limits_{i\in I(\bar x)}\mathcal T_{D_i}(F(\bar x))^{\circ\perp}
	\end{equation}
	and \cref{lem:primal_MPDC_LICQ}, the validity of MPDC-LICQ at $\bar x$ implies that
	\[
		\nabla F(\bar x)\R^n+\mathcal T_D(F(\bar x))^{\circ\perp}=\R^m
	\]
	holds. The latter condition is referred to as nondegeneracy in the setting where $D$ is convex, see
	\cite[Section~6.4.1]{BonnansShapiro2000}.
	Noting that $D$ is typically nonconvex in our setting, we would like to
	mention that a related conditions in the context of disjunctive programming can be found
	in \cite[Definition~3.6]{Gfrerer2014}.
\end{remark}
The following lemma will be important for our remaining considerations.
\begin{lemma}\label{lem:subspace_of_tangent_cone}
	For each feasible point $\bar x\in X$ of \eqref{eq:MPDC}, the following conditions hold:
	\begin{subequations}
		\begin{align}
			\label{eq:subspace_property1}
			&\mathcal T_D(F(\bar x))+\bigcap\limits_{i\in I(\bar x)}\mathcal T_{D_i}(F(\bar x))^{\circ\perp}
			\subset \mathcal T_D(F(\bar x)),\\
			\label{eq:subspace_property2}
			&\mathcal N_D(F(\bar x))
			\subset\sum\limits_{i\in I(\bar x)}\spa\widehat{\mathcal N}_{D_i}(F(\bar x)).
		\end{align}
	\end{subequations}
\end{lemma}
\begin{proof}
	Using \cref{lem:cones_to_disjunctive_structures} and the convexity of $D_1,\ldots,D_r$, we find
	\begin{align*}
		\mathcal T_D(F(\bar x))+\bigcap\limits_{i\in I(\bar x)}\mathcal T_{D_i}(F(\bar x))^{\circ\perp}
		&=\left(\bigcup\limits_{j\in I(\bar x)}\mathcal T_{D_j}(F(\bar x))\right)+\bigcap\limits_{i\in I(\bar x)}\mathcal T_{D_i}(F(\bar x))^{\circ\perp}\\
		&=\bigcup\limits_{j\in I(\bar x)}\left(\mathcal T_{D_j}(F(\bar x))
			+\bigcap\limits_{i\in I(\bar x)}\mathcal T_{D_i}(F(\bar x))^{\circ\perp}\right)\\
		&\subset\bigcup\limits_{j\in I(\bar x)}\left(\mathcal T_{D_j}(F(\bar x))+\mathcal T_{D_j}(F(\bar x))^{\circ\perp}\right)\\
		&=\bigcup\limits_{j\in I(\bar x)}\left(\mathcal T_{D_j}(F(\bar x))+\mathcal T_{D_j}(F(\bar x))\cap(-\mathcal T_{D_j}(F(\bar x)))\right)\\
		&=\bigcup\limits_{j\in I(\bar x)}\mathcal T_{D_j}(F(\bar x))=\mathcal T_D(F(\bar x))
	\end{align*}
	since we have $C+C\cap(-C)\subset C+C=C\subset C+C\cap(-C)$ for any closed, convex cone $C\subset\R^m$.
	This shows the validity of \eqref{eq:subspace_property1}. 
	For the proof of \eqref{eq:subspace_property2}, 
	we invoke \cref{lem:cones_to_disjunctive_structures} in order to see
	\begin{align*}
		\mathcal N_D(F(\bar x))
			&\subset\bigcup\limits_{i\in I(\bar x)}\widehat{\mathcal N}_{D_i}(F(\bar x))
			 \subset\sum\limits_{i\in I(\bar x)}\widehat{\mathcal N}_{D_i}(F(\bar x))
			 \subset\sum\limits_{i\in I(\bar x)}\spa\widehat{\mathcal N}_{D_i}(F(\bar x)).
	\end{align*}
	This already completes the proof. 
\end{proof}

We combine the above lemma with \cref{prop:subspaceCQ} and \cref{lem:primal_MPDC_LICQ}
in order to obtain the following result.
\begin{corollary}\label{cor:MPDC_LICQ_yields_explicit_formula_for_Frechet_normal_cone}
	Let $\bar x\in X$ be a feasible point of \eqref{eq:MPDC} where MPDC-LICQ is valid. 
	Then, NNAMCQ is valid for \eqref{eq:MPDC} at $\bar x$.
	Furthermore, we have
	\[
		\widehat{\mathcal{N}}_X(\bar x)=\nabla F(\bar x)^\top\widehat{\mathcal N}_D(F(\bar x)).
	\]
\end{corollary}
\begin{proof}
	The validity of NNAMCQ for \eqref{eq:MPDC} at $\bar x$ follows from \eqref{eq:subspace_property2} 
	and the definition of MPDC-LICQ.
	Particularly, MSCQ holds for \eqref{eq:MPDC} at $\bar x$.
	Now, we can combine the observation with \eqref{eq:subspace_property1}, 
	\cref{prop:subspaceCQ}, and \cref{lem:primal_MPDC_LICQ}
	in order to finish the proof.
\end{proof}

Clearly, our definition of MPDC-LICQ from \cref{def:MPDC-LICQ} is motivated by 
\cref{prop:subspaceCQ}. Thus, the main issue here is the choice of a \emph{reasonable}
subspace $L\subset\R^m$ such that the condition
\begin{equation}\label{eq:nondegeneracy}
		\mathcal T_D(F(\bar x))+L\subset\mathcal T_D(F(\bar x))
\end{equation}
holds for a fixed feasible point $\bar x\in X$ of \eqref{eq:MPDC}. 
As we have seen in \cref{lem:subspace_of_tangent_cone},
the subspace $\bigcap_{i\in I(\bar x)}\mathcal T_{D_i}(F(\bar x))^{\circ\perp}$ satisfies 
this condition while its annihilator is an upper approximation
of $\mathcal N_D(F(\bar x))$. We note that the validity of \eqref{eq:nondegeneracy} already implies
the relation $\mathcal T_D(F(\bar x))\cup L\subset\mathcal T_D(F(\bar x))$ which yields
$\widehat{\mathcal N}_D(F(\bar x))\cap L^\perp\supset\widehat{\mathcal N}_D(F(\bar x))$
by polarization 
and, thus, $L^\perp\supset \widehat{\mathcal N}_D(F(\bar x))$. 
Consequently, \cref{lem:cones_to_disjunctive_structures} shows that $L$ necessarily needs
to satisfy $L\subset(\bigcap_{i\in I(\bar x)}\mathcal T_{D_i}(F(\bar x))^\circ)^\perp$.
Due to \eqref{eq:subspace_candidates}, another reasonable candidate for the choice of
$L$ would be $\sum_{i\in I(\bar x)}\mathcal T_{D_i}(F(\bar x))^{\circ\perp}$.
However, considering e.g.\ $r=m:=2$, $F(\bar x):=0^2$, $D_1:=\R\times\{0\}$, and $D_2:=\{0\}\times\R^+$,
one can easily check that this subspace is still too large since it violates 
the condition \eqref{eq:nondegeneracy}.
Nevertheless, it might be possible that there is a subspace $L$ satisfying
\[
	\bigcap\limits_{i\in I(\bar x)}\mathcal T_{D_i}(F(\bar x))^{\circ\perp}
	\subsetneq
	L
	\subsetneq
	\sum\limits_{i\in I(\bar x)}\mathcal T_{D_i}(F(\bar x))^{\circ\perp}
\]
as well as \eqref{eq:nondegeneracy}. This way, the resulting
LICQ-type condition $\nabla F(\bar x)\R^n+L=\R^m$ would be less restrictive than
MPDC-LICQ from \cref{def:MPDC-LICQ}. However, it is not clear whether this
condition can be used to infer all the results of this paper which are mainly
valid under MPDC-LICQ.

Let us briefly interrelate the constraint qualification MPDC-LICQ with other prominent
constraint qualifications from disjunctive programming.
\begin{remark}\label{rem:MPDC_LICQ_yields_GACQ_and_GGCQ}
	Let $\bar x\in X$ be a feasible point of \eqref{eq:MPDC} where MPDC-LICQ is valid.
	Then, due to \cref{cor:MPDC_LICQ_yields_explicit_formula_for_Frechet_normal_cone}, we 
	obtain that the constraint qualifications NNAMCQ and MSCQ hold for
	\eqref{eq:MPDC} at $\bar x$ as well. 
	Particularly,  we obtain $\mathcal T_X(\bar x)=\mathcal L_X(\bar x)$ where
	$\mathcal L_X(\bar x)$
	denotes the linearization cone to $X$ at $\bar x$, see \eqref{eq:linearization_cone}. 
	In the literature of disjunctive
	programming, this condition is called \emph{generalized Abadie constraint qualification}
	(GACQ), see \cite[Definition~6]{FlegelKanzowOutrata2007}. 
	Furthermore, we obtain $\widehat{\mathcal N}_X(\bar x)=\mathcal L_X(\bar x)^\circ$ 
	by polarization, and the latter condition
	is called \emph{generalized Guignard constraint qualification} (GGCQ), see 
	\cite[Definition~6]{FlegelKanzowOutrata2007}.
\end{remark}

Now, it is possible to exploit \cref{prop:subspaceCQ} in order to derive necessary optimality
conditions of strong stationarity-type for \eqref{eq:MPDC}.
\begin{theorem}\label{thm:SSt_via_MPDC-LICQ}
	Let $\bar x\in\R^n$ be a locally optimal solution of \eqref{eq:MPDC} where MPDC-LICQ is valid.
	Then, there exists a uniquely determined multiplier
	$\lambda\in\R^m$ such that we have
	\[
		0^n=\nabla f(\bar x)+\nabla F(\bar x)^\top\lambda,\,
		\lambda\in\bigcap\limits_{i\in I(\bar x)}\widehat{\mathcal N}_{D_i}(F(\bar x)).
	\]
\end{theorem}
\begin{proof}
	Due to \cite[Theorem~6.12]{RockafellarWets1998}, we have $-\nabla f(\bar x)\in\widehat{\mathcal N}_X(\bar x)$.
	Invoking \cref{lem:cones_to_disjunctive_structures} and 
	\cref{cor:MPDC_LICQ_yields_explicit_formula_for_Frechet_normal_cone}, 
	we obtain
	\[
		\widehat{\mathcal N}_X(\bar x)=\nabla F(\bar x)^\top\widehat{\mathcal N}_D(F(\bar x))
			=\nabla F(\bar x)^\top\left[\bigcap\limits_{i\in I(\bar x)}\widehat{\mathcal N}_{D_i}(F(\bar x))\right],
	\]
	i.e.\ the postulated stationarity system possesses a solution.
	
	It remains to show that the associated multiplier is uniquely determined. 
	Therefore, assume that there are $\lambda^1,\lambda^2\in\bigcap_{i\in I(\bar x)}\widehat{\mathcal N}_{D_i}(F(\bar x))$
	satisfying $0^n=\nabla f(\bar x)+\nabla F(\bar x)^\top\lambda^s$, $s=1,2$.
	This yields $0^n=\nabla F(\bar x)^\top(\lambda^1-\lambda^2)$. Moreover, for each $i\in I(\bar x)$, we have
	\[
		\lambda^1-\lambda^2\in\widehat{\mathcal N}_{D_i}(F(\bar x))-\widehat{\mathcal N}_{D_i}(F(\bar x))
					=\spa\widehat{\mathcal N}_{D_i}(F(\bar x)).
	\]
	This yields $\lambda^1-\lambda^2\in\sum_{i\in I(\bar x)}\spa\widehat{\mathcal N}_{D_i}(F(\bar x))$, and by
	validity of MPDC-LICQ, $\lambda^1=\lambda^2$ follows. This completes the proof.
\end{proof}

Note that the multiplier $\lambda$ in \cref{thm:SSt_via_MPDC-LICQ} is chosen from the Fr\'{e}chet normal cone 
$\widehat{\mathcal N}_D(F(\bar x))$. Keeping \cite[Definition~1]{FlegelKanzowOutrata2007} in mind, this observation
justifies to call the above necessary optimality condition a strong stationarity-type condition. 
\begin{definition}\label{def:SSt}
	A feasible point $\bar x\in X$ of \eqref{eq:MPDC} is called strongly stationary (S-stationary for short)
	if any only if there exists a multiplier $\lambda\in\bigcap_{i\in I(\bar x)}\widehat{\mathcal N}_{D_i}(F(\bar x))$
	which satisfies $0^n=\nabla f(\bar x)+\nabla F(\bar x)^\top\lambda$.
\end{definition}

Some general considerations regarding S-stationary points of disjunctive programs can be found in 
\cite{FlegelKanzowOutrata2007,BenkoGfrerer2017,BenkoGfrerer2018}. 
We note that for prominent classes of disjunctive programs like MPCCs, MPVCs, CCMPs, and MPSCs, there exist 
respective strong stationarity notions which can be obtained by applying \cref{def:SSt} to the specific
problem setting, see \cref{sec:application}.
With the aid of \cref{ex:standard_nonlinear_programming}, it is easily seen that for \eqref{eq:NLP},
the S-stationarity system equals the classical Karush-Kuhn-Tucker conditions.

Due to \cref{rem:MPDC_LICQ_yields_GACQ_and_GGCQ}, the validity of MPDC-LICQ at $\bar x$
implies that the tangent cone to $X$ at $\bar x$ equals the associated linearization cone.
As we will see in the lemmas below, we also obtain derivability of $X$ at $\bar x$ 
as well as a nice representation of the second-order
tangent sets to $X$ at $\bar x$ in each direction $d\in\mathcal T_X(\bar x)$. 
\begin{lemma}\label{lem:MPDC_LICQ_yields_derivability}
	Let $\bar x\in X$ be a feasible point of \eqref{eq:MPDC} where MPDC-LICQ is valid.
	Then, $X$ is derivable at $\bar x$.
\end{lemma}
\begin{proof}
	Due to \cref{lem:cones_to_disjunctive_structures}, we obtain the inclusions
	\[
		\bigcap\limits_{i\in I(\bar x)}\mathcal T_{D_i}(F(\bar x))^{\circ\perp}
		\subset
		\bigcap\limits_{i\in I(\bar x)}\mathcal T_{D_i}(F(\bar x))
		=
		\bigcap\limits_{i\in I(\bar x)}\mathcal T_{D_i}^\textup{c}(F(\bar x))
		\subset
		\mathcal T_D^\textup{c}(F(\bar x)).
	\]
	Invoking \cref{lem:primal_MPDC_LICQ}, the validity of MPDC-LICQ yields
	\[
		\nabla F(\bar x)\R^n+\mathcal T_D^\textup{c}(F(\bar x))=\R^m.
	\]
	Thus, \cite[Theorem~4.3.3]{AubinFrankowska2009} can be applied in order to obtain
	\[
		\mathcal T_X^\flat(\bar x)=\left\{d\in\R^n\,\middle|\,\nabla F(\bar x)d\in\mathcal T^\flat_D(F(\bar x))\right\}.
	\]
	Since $D$ is derivable at $\bar x$, see \cref{lem:cones_to_disjunctive_structures}, this
	yields $\mathcal T_X^\flat(\bar x)=\mathcal L_X(\bar x)$. 
	Due to \cref{rem:MPDC_LICQ_yields_GACQ_and_GGCQ}, the validity of MPDC-LICQ also
	guarantees $\mathcal T_X(\bar x)=\mathcal L_X(\bar x)$, i.e.\
	$\mathcal T^\flat_X(\bar x)=\mathcal T_X(\bar x)$ follows,
	and this yields the claim.
\end{proof}
\begin{lemma}\label{lem:MPDC_LICQ_yields_parabolic_derivability}
	Let $\bar x\in X$ be a feasible point of \eqref{eq:MPDC} where MPDC-LICQ is valid.
	Then, for each $d\in\mathcal T_X(\bar x)$, we have
		\[
			\mathcal T^2_X(\bar x;d)=
				\left\{
					h\in\R^n\,\middle|\,
						\nabla F(\bar x)h+\nabla^2F(\bar x)[d,d]
							\in\mathcal T_{\mathcal T_D(F(\bar x))}(\nabla F(\bar x)d)
				\right\}.
		\]
	Furthermore, $\mathcal T^2_X(\bar x;d)$ is nonempty
	and $X$ is parabolically derivable at $\bar x$ in direction $d$.
\end{lemma}
\begin{proof}
	First, we note that the formula for the outer second-order tangent set follows from
	\cref{lem:outer_second_order_tangents_preimage} noting that
	$\mathcal T^2_D(F(\bar x);\nabla F(\bar x)d)=\mathcal T_{\mathcal T_D(F(\bar x))}(\nabla F(\bar x)d)$
	holds due to \cref{lem:disjunctive_sets_parabolically_derivable} while observing
	that the validity of MPDC-LICQ particularly yields that MSCQ is valid for \eqref{eq:MPDC}
	at $\bar x$. For later use, we would like to mention that this implies 
	$\mathcal T_X(\bar x)=\mathcal L_X(\bar x)$ as well, see \cref{rem:MPDC_LICQ_yields_GACQ_and_GGCQ}.
	
	Next, let us show that $\mathcal T^2_X(\bar x;d)$ is nonempty. 
	Due to validity of $d\in\mathcal L_X(\bar x)$, the set 
	$\mathcal T_{\mathcal T_D(F(\bar x))}(\nabla F(\bar x)d)=\mathcal T^2_D(F(\bar x);\nabla F(\bar x)d)$ 
	cannot be empty. We fix an arbitrary vector $r\in\mathcal T^2_D(F(\bar x);\nabla F(\bar x)d)$ and
	observe by means of \cref{lem:disjunctive_sets_parabolically_derivable} that
	\[
		\{r\}+\bigcap\limits_{i\in I(\bar x)}\mathcal T_{D_i}(F(\bar x))^{\circ\perp}
			\subset \mathcal T^2_D(F(\bar x);\nabla F(\bar x)d)
			=\mathcal T_{\mathcal T_D(F(\bar x))}(\nabla F(\bar x)d)
	\]
	holds true. By validity of MPDC-LICQ and \cref{lem:primal_MPDC_LICQ}, we now obtain
	\[
		\nabla F(\bar x)\R^n+\mathcal T_{\mathcal T_D(F(\bar x))}(\nabla F(\bar x)d)=\R^m.
	\]
	Particularly, we find some vector $h\in\R^n$ and some 
	$w\in\mathcal T_{\mathcal T_D(F(\bar x))}(\nabla F(\bar x)d)$
	such that we have $\nabla F(\bar x)h+w=\nabla^2F(\bar x)[d,d]$, i.e.\
	$-h\in\mathcal T^2_X(\bar x;d)$ is valid.
	
	Exploiting similar arguments as provided in the proof of \cref{lem:outer_second_order_tangents_preimage}
	while observing that $D$ is parabolically derivable due to
	\cref{lem:disjunctive_sets_parabolically_derivable}, we can show
	\[
		\mathcal T^{\flat,2}_X(\bar x;d)=
				\left\{
					h\in\R^n\,\middle|\,
						\nabla F(\bar x)h+\nabla^2F(\bar x)[d,d]
						\in\mathcal T_{\mathcal T_D(F(\bar x))}(\nabla F(\bar x)d)
				\right\},
	\]
	i.e.\ $X$ is parabolically derivable at $\bar x$ in direction $d\in\mathcal T_X(\bar x)$.
\end{proof}

\section{Second-order optimality conditions and MPDC-LICQ}\label{sec:second_order_conditions}

Recall that for each feasible point $\bar x\in X$ of \eqref{eq:MPDC}, $\mathcal L_X(\bar x)$ denotes the linearization
cone to $X$ at $\bar x$ and has been defined in \eqref{eq:linearization_cone}. 
For later use, we introduce the so-called critical cone to $X$ at $\bar x$ by means of
\begin{equation}\label{eq:critical_cone}
	\mathcal C_X(\bar x):=\left\{d\in\mathcal L_X(\bar x)\,\middle|\,\nabla f(\bar x)\cdot d\leq 0\right\}.
\end{equation}
Furthermore, we will exploit the so-called Lagrangian function $L\colon\R^n\times\R^m\to\R$ of \eqref{eq:MPDC}
which is given as stated below:
\[
	\forall (x,\lambda)\in\R^n\times\R^m\colon\quad
	L(x,\lambda):=f(x)+F(x)\cdot\lambda.
\]
Finally, let us introduce
\[
	S(\bar x):=\left\{\lambda\in\bigcap\nolimits_{i\in I(\bar x)}\widehat{\mathcal N}_{D_i}(F(\bar x))\,\middle|\,
						\nabla_xL(\bar x,\lambda)=0^n\right\},
\]
the set of all multipliers which solve the S-stationarity system associated with \eqref{eq:MPDC} at $\bar x$.
Clearly, $\bar x$ is an S-stationary point of \eqref{eq:MPDC} if and only if $S(\bar x)$ is nonempty.
\begin{lemma}\label{lem:critical_cone_and_SSt}
	Let $\bar x\in X$ be an S-stationary point of \eqref{eq:MPDC}. Then, we have
	\[
		\forall\lambda\in S(\bar x)\colon\quad 
		\mathcal C_X(\bar x)=\left\{d\in\R^n\,\middle|\,\nabla F(\bar x)d\in\mathcal T_D(F(\bar x))\cap\{\lambda\}^\perp\right\}.
	\]
\end{lemma}
\begin{proof}
	For each $d\in\mathcal C_X(\bar x)$ and $\lambda\in S(\bar x)$, we obtain
	\[
		0\geq \nabla f(\bar x)\cdot d=(-\nabla F(\bar x)^\top\lambda)\cdot d
			=-(\underbrace{\nabla F(\bar x)d}_{\in\mathcal T_D(F(\bar x))})\cdot\lambda\geq 0
	\]
	from $\lambda\in\bigcap_{i\in I(\bar x)}\widehat{\mathcal N}_{D_i}(F(\bar x))
	=\widehat{\mathcal N}_D(F(\bar x))=\mathcal T_D(F(\bar x))^\circ$, see \cref{lem:cones_to_disjunctive_structures}. 
	This yields $\nabla F(\bar x)d\in\{\lambda\}^\perp$ and shows the inclusion $\subset$. 
	
	If, on the other hand, $d\in\R^n$ satisfies $\nabla F(\bar x)d\in\mathcal T_D(F(\bar x))\cap\{\lambda\}^\perp$ for
	some $\lambda\in S(\bar x)$, then we have $d\in\mathcal L_X(\bar x)$ by definition of the linearization cone and
	\[
		0=(\nabla F(\bar x)d)\cdot\lambda=(\nabla F(\bar x)^\top\lambda)\cdot d=-\nabla f(\bar x)\cdot d
	\]
	which yields $d\in\mathcal C_X(\bar x)$.
\end{proof}

Using the theory on second-order tangent sets provided earlier, we are now in position to state
a second-order necessary optimality condition for \eqref{eq:MPDC} under validity of MPDC-LICQ.
Thus, our approach is closely related to the approaches used in 
\cite{BonnansShapiro2000,ChristofWachsmuth2018,Penot1998,RockafellarWets1998}
for the derivation of second-order necessary optimality conditions for different classes of
mathematical programs in the finite- and infinite-dimensional setting.
It seems to be worth mentioning that, in contrast to \cite[Theorem~4.3]{HoheiselKanzow2007}
where a second-order necessary optimality conditions for MPVCs is shown, 
we do not use an implicit function argument for our proof.
Some parts of the upcoming theorem's proof are inspired by
\cite[Lemma~5.8]{ChristofWachsmuth2018}. 
\begin{theorem}\label{thm:SONC}
	Let $\bar x\in X$ be a locally optimal solution of \eqref{eq:MPDC}
	where MPDC-LICQ is valid. 
	Then, we have
	\[
		\forall d\in\mathcal C_X(\bar x)\colon
		\quad
		d^\top\nabla^2_{xx}L(\bar x,\bar\lambda)d\geq 0
	\]
	where
	$\bar\lambda\in S(\bar x)$
	is the uniquely determined multiplier which solves the 
	S-stationarity system associated with $\bar x$, see \cref{thm:SSt_via_MPDC-LICQ}.
\end{theorem}
\begin{proof}
	First, we will prove the correctness of 
	\begin{equation}\label{eq:general_second_order_necesarry_condition}
		\forall d\in\mathcal C_X(\bar x)\,\forall h\in\mathcal T^2_X(\bar x;d)\colon\quad
		\nabla f(\bar x)\cdot h+d^\top\nabla^2f(\bar x)d\geq 0.
	\end{equation}
	Therefore, fix $d\in\mathcal C_X(\bar x)$ and $h\in\mathcal T^2_X(\bar x;d)$.
	Then, we find sequences $\{t_k\}_{k\in\N}\subset \R_+$ and $\{h_k\}_{k\in\N}\subset\R^n$
	such that $t_k\downarrow 0$, $h_k\to h$, and $\bar x+t_kd+\tfrac12t_k^2h_k\in X$ for all
	$k\in\N$. Performing a second-order Taylor expansion of $f$ at $\bar x$ yields
	\[
		f(\bar x+t_kd+\tfrac12t_k^2h_k)=f(\bar x)+t_k\nabla f(\bar x)\cdot d
			+\tfrac12t_k^2(\nabla f(\bar x)\cdot h_k+d^\top \nabla ^2f(\bar x)d)+\oo(t_k^2)
	\]
	for all $k\in\N$. Noting that we have $f(\bar x+t_kd+\tfrac12t_k^2h_k)\geq f(\bar x)$
	for sufficiently large $k\in\N$
	from the local optimality of $\bar x$ for \eqref{eq:MPDC} while 
	$\nabla f(\bar x)\cdot d\leq 0$ holds by definition of the critical cone, we obtain
	\[
		0\leq\frac2{t_k^2}
			\left(
				f(\bar x+t_kd+\tfrac12t_k^2h_k)-f(\bar x)
				-t_k\nabla f(\bar x)\cdot d
			\right)
		=\nabla f(\bar x)\cdot h_k+d^\top\nabla^2f(\bar x)d+2\frac{\oo(t_k^2)}{t_k^2}
	\]
	for sufficiently large $k\in\N$. Thus, taking the limit $k\to\infty$ yields
	\eqref{eq:general_second_order_necesarry_condition}.
	
	Due to validity of MPDC-LICQ, \eqref{eq:general_second_order_necesarry_condition}
	implies that
	\[
		\inf\left\{
			\nabla f(\bar x)\cdot h\,\middle|\,
				\nabla F(\bar x)h\in\mathcal T_{\mathcal T_D(F(\bar x))}(\nabla F(\bar x)d)
				-\{\nabla^2F(\bar x)[d,d]\}\right\}
		+d^\top\nabla^2f(\bar x)d\geq 0
	\]
	holds true for all $d\in\mathcal C_X(\bar x)$, see \cref{lem:MPDC_LICQ_yields_parabolic_derivability}. 
	By definition of S-stationarity, we
	have $\nabla f(\bar x)=-\nabla F(\bar x)^\top\bar\lambda$ which yields
	\begin{equation}\label{eq:SONC_via_directional_curvature_functional}
		\inf\left\{
			-\bar\lambda\cdot w\,\middle|\begin{aligned}
				&w\in\mathcal T_{\mathcal T_D(F(\bar x))}(\nabla F(\bar x)d)
					-\{\nabla^2F(\bar x)[d,d]\}\\
				&w\in\nabla F(\bar x)\R^n
				\end{aligned}
				\right\}
		+d^\top\nabla^2f(\bar x)d\geq 0
	\end{equation}
	for each $d\in\mathcal C_X(\bar x)$.
	
	Due to validity of MPDC-LICQ, for each 
	$w\in\mathcal T_{\mathcal T_D(F(\bar x))}(\nabla F(\bar x)d)-\{\nabla^2F(\bar x)[d,d]\}$,
	we find $v\in\R^n$ and 
	$\ell\in\bigcap_{i\in I(\bar x)}\mathcal T_{D_i}(F(\bar x))^{\circ\perp}$ such that
	$w=\nabla F(\bar x)v-\ell$ holds true, see \cref{lem:primal_MPDC_LICQ}. Noting that we have
	\[
		\bigcap\limits_{i\in I(\bar x)}\mathcal T_{D_i}(F(\bar x))^{\circ\perp}
		=\bigcap\limits_{i\in I(\bar x)}\widehat{\mathcal N}_{D_i}(F(\bar x))^\perp
		\subset\left(\bigcap\limits_{i\in I(\bar x)}\widehat{\mathcal N}_{D_i}(F(\bar x))\right)^\perp
		=\widehat{\mathcal N}_D(F(\bar x))^\perp,
	\]
	see \cref{lem:cones_to_disjunctive_structures},
	the relation $\bar\lambda\cdot\ell=0$ follows from
	$\bar\lambda\in\widehat{\mathcal N}_D(F(\bar x))$. Furthermore, we infer
	\begin{align*}
		\nabla F(\bar x)v=w+\ell
		&\in\mathcal T_{\mathcal T_D(F(\bar x))}(\nabla F(\bar x)d)
			-\{\nabla^2F(\bar x)[d,d]\}
			+\bigcap\limits_{i\in I(\bar x)}\mathcal T_{D_i}(F(\bar x))^{\circ\perp}\\
		&=\mathcal T^2_D(F(\bar x);\nabla F(\bar x)d)
			+\bigcap\limits_{i\in I(\bar x)}\mathcal T_{D_i}(F(\bar x))^{\circ\perp}
			-\{\nabla^2 F(\bar x)[d,d]\}\\
		&\subset \mathcal T^2_D(F(\bar x);\nabla F(\bar x)d)
			-\{\nabla^2F(\bar x)[d,d]\}\\
		&=\mathcal T_{\mathcal T_D(F(\bar x))}(\nabla F(\bar x)d)
			-\{\nabla^2F(\bar x)[d,d]\}
	\end{align*}
	from \cref{lem:disjunctive_sets_parabolically_derivable}.
	Summarizing these considerations, we have shown $\nabla F(\bar x)v\in\nabla F(\bar x)\R^n$ and 
	$\nabla F(\bar x)v\in\mathcal T_{\mathcal T_D(F(\bar x))}(\nabla F(\bar x)d)-\{\nabla^2F(\bar x)[d,d]\}$.
	Furthermore, we obtain the relation
	$-\bar\lambda\cdot w=-\bar\lambda(\nabla F(\bar x)v-\ell)=-\bar\lambda\cdot(\nabla F(\bar x)v)$.
	This leads to
	\begin{align*}
		&\inf\left\{
			-\bar\lambda\cdot w\,\middle|\begin{aligned}
				&w\in\mathcal T_{\mathcal T_D(F(\bar x))}(\nabla F(\bar x)d)
					-\{\nabla^2F(\bar x)[d,d]\}\\
				&w\in\nabla F(\bar x)\R^n
				\end{aligned}
				\right\}\\
		&\qquad
		\leq\inf\left\{
			-\bar\lambda\cdot w\,\middle|\,
				w\in\mathcal T_{\mathcal T_D(F(\bar x))}(\nabla F(\bar x)d)
					-\{\nabla^2F(\bar x)[d,d]\}
		\right\}.
	\end{align*}
	The converse inequality, however, is trivial. Thus, equality holds for
	the optimal values of the above programs and we obtain
	\begin{equation}\label{eq:SONC_with_trivial_infimum}
		\inf\left\{
			-\bar\lambda\cdot w\,\middle|\,
				w\in\mathcal T_{\mathcal T_D(F(\bar x))}(\nabla F(\bar x)d)
					-\{\nabla^2F(\bar x)[d,d]\}
				\right\}
		+d^\top\nabla^2f(\bar x)d\geq 0
	\end{equation}
	for each $d\in\mathcal C_X(\bar x)$ from \eqref{eq:SONC_via_directional_curvature_functional}.
	Clearly, we have
	\begin{align*}
		&\inf\left\{
			-\bar\lambda\cdot w\,\middle|\,
				w\in\mathcal T_{\mathcal T_D(F(\bar x))}(\nabla F(\bar x)d)
					-\{\nabla^2F(\bar x)[d,d]\}
				\right\}\\
		&\qquad =\inf\left\{-\bar\lambda\cdot w\,\middle|\,
				w\in\mathcal T_{\mathcal T_D(F(\bar x))}(\nabla F(\bar x)d)
			\right\}
			+\bar\lambda\cdot\nabla^2F(\bar x)[d,d].
	\end{align*}
	Finally, we note that
	\begin{align*}
		&\mathcal T_{\mathcal T_D(F(\bar x))}(\nabla F(\bar x)d)
		=\bigcup\limits_{i\in I(\bar x)}\mathcal T_{\mathcal T_{D_i}(F(\bar x))}(\nabla F(\bar x)d)\\
		&\qquad=\bigcup\limits_{i\in I(\bar x)}\mathcal T_{D_i}(F(\bar x))-\cone\{\nabla F(\bar x)d\}
		=\mathcal T_D(F(\bar x))-\cone\{\nabla F(\bar x)d\}
	\end{align*}
	holds true invoking \cref{lem:cones_to_disjunctive_structures} while
	noticing that the sets $\mathcal T_{D_i}(F(\bar x))$, $i\in I(\bar x)$, are
	closed, convex, polyhedral cones. Thus, for each 
	$w\in\mathcal T_{\mathcal T_D(F(\bar x))}(\nabla F(\bar x)d)$, we find a vector
	$r\in\mathcal T_D(F(\bar x))$ and $\alpha\geq 0$ such that $w=r-\alpha\nabla F(\bar x)d$
	holds. Recalling $\bar\lambda\in\widehat{\mathcal N}_D(F(\bar x))$ and 
	$d\in\mathcal C_X(\bar x)$, we have
	\[
		-\bar\lambda\cdot w=-\bar\lambda\cdot(r-\alpha\nabla F(\bar x)d)
		\geq\alpha(\nabla F(\bar x)^\top\bar\lambda)\cdot d
		=\alpha(-\nabla f(\bar x))\cdot d\geq 0
	\]
	by definition of S-stationarity, i.e.\
	\[
		\inf\left\{-\bar\lambda\cdot w\,\middle|\,
				w\in\mathcal T_{\mathcal T_D(F(\bar x))}(\nabla F(\bar x)d)
			\right\}
			+\bar\lambda\cdot\nabla^2F(\bar x)[d,d]=\bar\lambda\cdot\nabla^2F(\bar x)[d,d]
	\]
	follows for each $d\in\mathcal C_X(\bar x)$. Combining this with the above
	arguments, the desired result follows from \eqref{eq:SONC_with_trivial_infimum}
	by definition of the Lagrangian function. This completes the proof.
\end{proof}

The above result can be seen as a particular instance of \cite[Theorem~3.3]{Gfrerer2014}
where a second-order necessary optimality condition for \eqref{eq:MPDC} has been
derived using a completely different approach via the variational concepts of 
the directional limiting normal cone and directional metric subregularity.
One can easily check that by demanding validity of MPDC-LICQ at a given local 
minimizer of \eqref{eq:MPDC}, the assumptions of \cite[Theorem~3.3]{Gfrerer2014}
hold as well, i.e.\ the assumptions of \cref{thm:SONC} are more restrictive. 
On the other hand, one has to mention that checking validity of MPDC-LICQ and 
noting that this implies that there is only one S-stationary multiplier, the second-
order necessary optimality condition from \cref{thm:SONC} seems to be much easier
to verify than the one from \cite{Gfrerer2014}. 

Next, we state a second-order sufficient optimality condition for \eqref{eq:MPDC}.
Although this result follows from \cite[Theorem~3.21]{Gfrerer2014}, we provide a completely elementary
and simple proof here which generalizes a well-known strategy which has been used to verify second-order
sufficient optimality conditions for NLPs, MPCCs, MPVCs, and CCMPs in the past.
\begin{theorem}\label{thm:SOSC}
	Let $\bar x\in X$ be an S-stationary point of \eqref{eq:MPDC} where the condition
	\begin{equation}\label{eq:MPDC-SOSC}
		\forall d\in\mathcal C_X(\bar x)\setminus\{0^n\}\,\exists\lambda\in S(\bar x)\colon\quad
		d^\top\nabla^2_{xx}L(\bar x,\lambda)d>0
	\end{equation}
	holds. Then, there are constants $\varepsilon>0$ and $C>0$ such that the following quadratic-growth-condition
	is valid:
	\[
		\forall x\in X\cap\mathbb B^\varepsilon(\bar x)\colon\quad
		f(x)\geq f(\bar x)+C\norm{x-\bar x}{2}^2.
	\]
	Particularly, $\bar x$ is a strict local minimizer of \eqref{eq:MPDC}.
\end{theorem}
\begin{proof}
	Assume on the contrary that there is a sequence $\{x_k\}_{k\in\N}\subset X$ converging to $\bar x$ such that
	\[
		\forall k\in\N\colon\quad
		f(x_k)<f(\bar x)+\tfrac1k\norm{x_k-\bar x}{2}^2
	\]
	holds true. Set $t_k:=\norm{x_k-\bar x}{2}>0$ and observe that  $\{(x_k-\bar x)/t_k\}_{k\in\N}$
	is bounded. We assume w.l.o.g.\ that $(x_k-\bar x)/t_k\to d$ holds for some $d\in\R^n\setminus\{0^n\}$. By construction, 
	$d\in\mathcal T_X(\bar x)\subset\mathcal L_X(\bar x)$ is guaranteed. 
	For each $k\in\N$, we find $\xi_k\in\conv\{\bar x;x_k\}$ which satisfies $f(x_k)-f(\bar x)=\nabla f(\xi_k)\cdot(x_k-\bar x)$
	by means of the mean value theorem. Dividing by $t_k$ and taking the limit $k\to\infty$ while observing that 
	$\nabla f\colon\R^n\to\R^n$ is continuous, we have
	\[
		\nabla f(\bar x)\cdot d=\lim\limits_{k\to\infty}\nabla f(\xi_k)\cdot\frac{x_k-\bar x}{t_k}
			=\lim\limits_{k\to\infty}\frac{f(x_k)-f(\bar x)}{t_k}
			\leq\lim\limits_{k\to\infty}\tfrac1k\norm{x_k-\bar x}{2}=0.
	\]
	This yields $d\in\mathcal C_X(\bar x)\setminus\{0^n\}$.
	
	Choose $\lambda\in S(\bar x)$ arbitrarily. Then, we have $\lambda\in\bigcap_{i\in I(\bar x)}\widehat{\mathcal N}_{D_i}(F(\bar x))$.
	For sufficiently large $k\in\N$, $I(x_k)\subset I(\bar x)$ holds true. Thus, for sufficiently large $k\in\N$ and $i\in I(x_k)$,
	we have $\lambda\in \widehat{\mathcal N}_{D_i}(F(\bar x))=(D_i-\{F(\bar x)\})^\circ$ which shows $(F(x_k)-F(\bar x))\cdot\lambda\leq 0$.
	This yields
	\[
		f(\bar x)>f(x_k)-\tfrac1k\norm{x_k-\bar x}{2}^2\geq f(x_k)+(F(x_k)-F(\bar x))\cdot\lambda-\tfrac1k\norm{x_k-\bar x}{2}^2
	\]
	for sufficiently large $k\in\N$. Rearranging some terms and applying Taylor's theorem, we derive
	\begin{align*}
		L(\bar x,\lambda)
			&>L(x_k,\lambda)-\tfrac1k\norm{x_k-\bar x}{2}^2\\
			&=L(\bar x,\lambda)+\nabla_xL(\bar x,\lambda)(x_k-\bar x)
				+\tfrac12(x_k-\bar x)^\top\nabla^2_{xx}L(\bar x,\lambda)(x_k-\bar x)+\oo(\norm{x_k-\bar x}{2}^2).
	\end{align*}
	Now, we exploit the choice $\lambda\in S(\bar x)$ in order to infer
	\[
		0>\tfrac12(x_k-\bar x)^\top\nabla^2_{xx}L(\bar x,\lambda)(x_k-\bar x)+\oo(\norm{x_k-\bar x}{2}^2)
	\]
	for sufficiently large $k\in\N$. Division by $t_k^2$ and taking the limit $k\to\infty$ yield
	\[
		0\geq \tfrac12 d^\top\nabla^2_{xx}L(\bar x,\lambda)d
	\]
	which contradicts the theorem's assumptions since we have shown $d\in\mathcal C_X(\bar x)\setminus\{0^n\}$ 
	while $\lambda\in S(\bar x)$ was arbitrarily chosen. This completes the proof.
\end{proof}

The above result justifies the following definition.
\begin{definition}\label{def:MPDC-SOSC}
	Let $\bar x\in X$ be an S-stationary point of \eqref{eq:MPDC}.
	Then, the MPDC-tailored \emph{second-order sufficient condition} 
	(MPDC-SOSC for short) holds at $\bar x$
	if and only if \eqref{eq:MPDC-SOSC} is valid.
\end{definition}

The upcoming considerations will show that S-stationary points of \eqref{eq:MPDC}, 
where both MPDC-LICQ and MPDC-SOSC are valid, are locally isolated
w.r.t.\ primal \emph{and} dual variables. 
This property does not generally follow from the second-order growth
condition as \cite[Example~4.1]{GuoLinYe2013}, which has been stated in the
context of MPCCs, indicates.
For the validation of the upcoming result, 
we generalize the proof of 
\cite[Theorem~4.1]{GuoLinYe2013}.
\begin{theorem}\label{thm:stability_of_S_stationary_points}
	Let $\bar x\in X$ be an S-stationary point of \eqref{eq:MPDC} where MPDC-LICQ and MPDC-SOSC are valid.
	Then, there is some $\varepsilon>0$ such that we have
	\begin{equation}\label{eq:primal_dual_isolatedness}
		\forall x\in X\cap\mathbb B^\varepsilon(\bar x)\colon\quad
		\lambda\in S(x)\,\Longrightarrow\,x=\bar x,\,\lambda=\bar\lambda
	\end{equation}
	where $\bar\lambda$ is the uniquely determined vector from $S(\bar x)$.
\end{theorem}
\begin{proof}
	Due to validity of MPDC-LICQ, the S-stationarity multiplier $\bar\lambda$ associated with
	$\bar x$ is indeed uniquely determined, see \cref{thm:SSt_via_MPDC-LICQ}.
	Assume on the contrary, that we can find a sequence $\{x_k\}_{k\in\N}\subset X\setminus\{\bar x\}$ of feasible and
	S-stationary points of \eqref{eq:MPDC} converging to $\bar x$. 
	Then, we find $\lambda_k\in S(x_k)$ for each $k\in\N$. 
	
	Suppose that $\{\lambda_k\}_{k\in\N}$ is not bounded, i.e.\ we can assume 
	w.l.o.g.\ that $\norm{\lambda_k}{2}\to\infty$ holds 
	as $k\to\infty$. Thus, we can define $\tilde\lambda_k:=\lambda_k/\norm{\lambda_k}{2}$
	for sufficiently large $k\in\N$ and due to the boundedness of $\{\tilde\lambda_k\}_{k\in\N}$,
	we may assume w.l.o.g.\ that this sequence converges to some nonvanishing vector $\tilde\lambda\in\R^m$.
	Furthermore, we have
	\[
		\nabla F(\bar x)^\top\tilde\lambda
			=\lim\limits_{k\to\infty}\nabla F(x_k)^\top\tilde\lambda_k
			=\lim\limits_{k\to\infty}\frac{1}{\norm{\lambda_k}{2}}
				\underbrace{\left(\nabla f(x_k)+\nabla F(x_k)^\top\lambda_k\right)}_{=0^n}
			=0^n
	\]
	by continuity of $\nabla f\colon\R^n\to\R^n$ and $\nabla F\colon\R^n\to\R^{m\times n}$ as well as 
	$\lambda_k\in S(x_k)$.
	On the other hand, the inclusion $I(x_k)\subset I(\bar x)$ is valid for all sufficiently large $k\in\N$
	and, clearly, $I(x_k)\neq\varnothing$ is true as well since $x_k$ is feasible to \eqref{eq:MPDC}
	for each $k\in\N$. Noting that there are only finitely many indices in $I(\bar x)$, there must
	exist some $i_0\in I(\bar x)$ such that $i_0\in I(x_{k_l})$ for all $l\in\N$ 
	holds along a subsequence $\{x_{k_l}\}_{l\in\N}$
	of $\{x_k\}_{k\in\N}$. The definition of S-stationarity and the fact that the Fr\'{e}chet normal
	cone is a cone yield $\tilde \lambda_{k_l}\in\widehat{\mathcal N}_{D_{i_0}}(F(x_{k_l}))$. 
	Now, the continuity of $F$ can be used to infer 
	\[
		\tilde\lambda\in\mathcal N_{D_{i_0}}(F(\bar x))
			=\widehat{\mathcal N}_{D_{i_0}}(F(\bar x))
			\subset\sum\limits_{i\in I(\bar x)}\spa\widehat{\mathcal N}_{D_i}(F(\bar x)).
	\]
	Keeping $\nabla F(\bar x)^\top\tilde\lambda=0^n$ and $\tilde\lambda\neq 0^m$ in mind, this contradicts
	MPDC-LICQ.
	
	Due to the above arguments, we may assume w.l.o.g.\ that $\{\lambda_k\}_{k\in\N}$ converges to some $\lambda\in\R^m$.
	Similar arguments as above show the existence of $i_0\in I(\bar x)$ such that 
	$\lambda\in\widehat{\mathcal N}_{D_{i_0}}(F(\bar x))$ holds true. 
	Moreover, from $\nabla f(x_k)+\nabla F(x_k)^\top\lambda_k=0^n$ we obtain $\nabla f(\bar x)+\nabla F(\bar x)^\top\lambda=0^n$
	since $f$ and $F$ possess continuous derivatives. Keeping $\nabla f(\bar x)+\nabla F(\bar x)^\top\bar\lambda=0^n$ in mind,
	we derive $\nabla F(\bar x)^\top(\bar\lambda-\lambda)=0^n$. Moreover, 
	\begin{align*}
		\bar\lambda-\lambda
			&\in\left(\bigcap\limits_{i\in I(\bar x)}\widehat{\mathcal N}_{D_i}(F(\bar x))\right)
				-\widehat{\mathcal N}_{D_{i_0}}(F(\bar x))\\
			&\subset\widehat{\mathcal N}_{D_{i_0}}(F(\bar x))-\widehat{\mathcal N}_{D_{i_0}}(F(\bar x))
			 =\spa\widehat{\mathcal N}_{D_{i_0}}(F(\bar x))
			 \subset\sum\limits_{i\in I(\bar x)}\spa\widehat{\mathcal N}_{D_i}(F(\bar x))	
	\end{align*}
	follows, and by validity of MPDC-LICQ, $\lambda=\bar\lambda$ is obtained.
	
	We set $t_k:=\norm{x_k-\bar x}{2}>0$ and observe that $\{(x_k-\bar x)/t_k\}_{k\in\N}$ is a bounded sequence
	that converges w.l.o.g.\ to some nonvanishing direction $d\in\R^n$. Since $\{x_k\}_{k\in\N}\subset X$ holds, we
	infer $d\in\mathcal T_X(\bar x)\setminus\{0^n\}\subset\mathcal L_X(\bar x)\setminus\{0^n\}$.
	From $\lambda_k\in\widehat{\mathcal N}_D(F(x_k))$ for all $k\in\N$,
	$x_k\to\bar x$, and $\lambda_k\to\bar\lambda$, we obtain 
	$\lambda_k\cdot(F(x_k)-F(\bar x))=\bar\lambda\cdot(F(x_k)-F(\bar x))=0$ for all sufficiently large $k\in\N$,
	see \cref{lem:stability_property}. 
	This yields
	\begin{align*}
		(\nabla F(\bar x)d)\cdot\bar\lambda
		&=\left(\lim\limits_{k\to\infty}\frac{\nabla F(\bar x)(x_k-\bar x)}{t_k}\right)\cdot\bar\lambda
		=\lim\limits_{k\to\infty}\frac{(F(x_k)-F(\bar x))\cdot\bar\lambda}{t_k}=0,
	\end{align*}
	i.e.\ $F(\bar x)d\in\{\bar\lambda\}^\perp$ holds true. By means of \cref{lem:critical_cone_and_SSt}, we deduce
	$d\in\mathcal C_X(\bar x)\setminus\{0^n\}$.
	
	For each $k\in\N$, let us define a continuously differentiable function $\varphi_k\colon[0,1]\to\R$ by means of
	\begin{align*}
		\forall s\in[0,1]\colon\quad \varphi_k(s)&:=
			\nabla _xL((1-s)(\bar x,\bar\lambda)+s(x_k,\lambda_k))\cdot(x_k-\bar x)\\
			&\qquad -L((1-s)\bar x+sx_k,\lambda_k)+L((1-s)\bar x+sx_k,\bar\lambda).
	\end{align*}
	Due to the above remarks, we have
	\begin{align*}
		\varphi_k(0)
		&=
		\nabla_xL(\bar x,\bar\lambda)\cdot(x_k-\bar x)-L(\bar x,\lambda_k)+L(\bar x,\bar\lambda)\\
		&=
		(\bar\lambda-\lambda_k)\cdot F(\bar x)=(\bar\lambda-\lambda_k)\cdot F(x_k)\\
		&=
		\nabla_xL(x_k,\lambda_k)\cdot(x_k-\bar x)-L(x_k,\lambda_k)+L(x_k,\bar\lambda)
		=\varphi_k(1)
	\end{align*}
	for sufficiently large $k\in\N$. Due to $\varphi_k(0)=\varphi_k(1)$, we can apply Rolle's theorem
	in order to obtain the existence of $s_k\in(0,1)$ such that
	\begin{align*}
		0&=\varphi'_k(s_k)\\
		&=(x_k-\bar x)^\top\nabla^2_{xx}L((1-s_k)(\bar x,\bar\lambda)+s_k(x_k,\lambda_k))(x_k-\bar x)\\
		&\qquad	+(x_k-\bar x)^\top\nabla^2_{x,\lambda}L((1-s_k)(\bar x,\bar\lambda)+s_k(x_k,\lambda_k))(\lambda_k-\bar\lambda)\\
		&\qquad -\nabla_xL((1-s_k)\bar x+s_kx_k,\lambda_k)\cdot (x_k-\bar x)
			+\nabla_xL((1-s_k)\bar x+s_kx_k,\bar\lambda)\cdot (x_k-\bar x)\\
		&=(x_k-\bar x)^\top\nabla^2_{xx}L((1-s_k)(\bar x,\bar\lambda)+s_k(x_k,\lambda_k))(x_k-\bar x)\\
		&\qquad +(x_k-\bar x)\cdot[\nabla F((1-s_k)\bar x+s_kx_k)^\top(\lambda_k-\bar \lambda)]\\
		&\qquad -[\nabla F((1-s_k)\bar x+s_kx_k)^\top(\lambda_k-\bar\lambda)]\cdot(x_k-\bar x)\\
		&=(x_k-\bar x)^\top\nabla^2_{xx}L((1-s_k)(\bar x,\bar\lambda)+s_k(x_k,\lambda_k))(x_k-\bar x)
	\end{align*}
	holds for all $k\in\N$ which are sufficiently large. Next, we observe that the relation
	$(1-s_k)(\bar x,\bar\lambda)+s_k(x_k,\lambda_k)\to(\bar x,\bar\lambda)$ holds true as $k\to\infty$.
	From above, it follows
	\[
		0=\left(\frac{x_k-\bar x}{t_k}\right)^\top\nabla^2_{xx}L((1-s_k)(\bar x,\bar\lambda)+s_k(x_k,\lambda_k))
			\left(\frac{x_k-\bar x}{t_k}\right)
	\]
	for sufficiently large $k\in\N$, i.e.\ taking the limit $k\to\infty$ yields 
	$0=d^\top\nabla^2_{xx}L(\bar x,\bar\lambda)d$. This, however, contradicts the validity of MPDC-SOSC
	since we already verified that $d\in\mathcal C_X(\bar x)\setminus\{0^n\}$ holds true.
	Thus, the proof is completed.
\end{proof}

\section{Consequences for certain classes of disjunctive programs}\label{sec:application}

In this section, we are going to apply the obtained results to some prominent classes of disjunctive programs,
namely MPCCs, MPVCs, CCMPs, and MPSCs in order to check how the above theory relates to existing results in
the available literature on these problem classes.
Throughout the section, we consider twice continuously differentiable functions $f\colon\R^n\to\R$,
$g\colon\R^n\to\R^p$, $h\colon\R^n\to\R^q$, and $G,H\colon\R^n\to\R^l$. The component mappings of $g$, $h$,
$G$, and $H$ will be denoted by $g_j\colon\R^n\to\R$, $j=1,\ldots,p$, $h_j\colon\R^n\to\R$, $j=1,\ldots,q$,
and $G_j,H_j\colon\R^n\to\R$, $j=1,\ldots,l$, respectively.

\subsection{Application to MPCCs}\label{sec:MPCC}
A \emph{mathematical program with complementarity constraints} is an optimization problem of the form
\begin{equation}\label{eq:MPCC}\tag{MPCC}
	\begin{aligned}
		f(x)&\,\rightarrow\,\min &&&\\
		g_j(x)&\,\leq\,0&\quad&j=1,\ldots,p&\\
		h_j(x)&\,=\,0&&j=1,\ldots,q&\\
		0\,\leq\,G_j(x)\,\perp\,H_j(x)&\,\geq\,0&&j=1,\ldots,l.&
	\end{aligned}
\end{equation}
Due to the frequent appearance of \eqref{eq:MPCC} as an abstract model of real-world applications,
this problem class has been studied intensively from the theoretical and numerical point of view
during the last two decades, see e.g.\
\cite{Gfrerer2014,HoheiselKanzowSchwartz2013,LuoPangRalph1996,OutrataKocvaraZowe1998,ScheelScholtes2000,Ye2005}
and the references therein. 

In order to transfer \eqref{eq:MPCC} into a program of type \eqref{eq:MPDC}, we introduce the sets $S^\text{CC}_1:=\R_+\times\{0\}$ and
$S^\text{CC}_2:=\{0\}\times\R_+$ as well as $\mathcal J:=\{1,2\}^l$. Next, we set
\[
	\forall \alpha\in\mathcal J\colon\quad D^\text{CC}_\alpha:=\R^p_-\times\{0^q\}\times\prod\nolimits_{j=1}^lS_{\alpha_j}^\text{CC}
\]
and $D^\text{CC}:=\bigcup_{\alpha\in\mathcal J}D_\alpha^\text{CC}$. Furthermore, we introduce $F\colon\R^n\to\R^{p+q+2l}$ by means of
\begin{equation}\label{eq:definition_constraint_map}
	\forall x\in\R^n\colon\quad
	F(x):=\begin{bmatrix}
		g(x)^\top&h(x)^\top&G_1(x)&H_1(x)&\dots&G_l(x)&H_l(x)
	\end{bmatrix}^\top.
\end{equation}
Then, the feasible set of \eqref{eq:MPCC} is given by $X^\text{CC}:=\{x\in\R^n\,|\,F(x)\in D^\text{CC}\}$,
see \cref{fig:MPCC} for an illustration.
\begin{figure}[h]\centering
\includegraphics{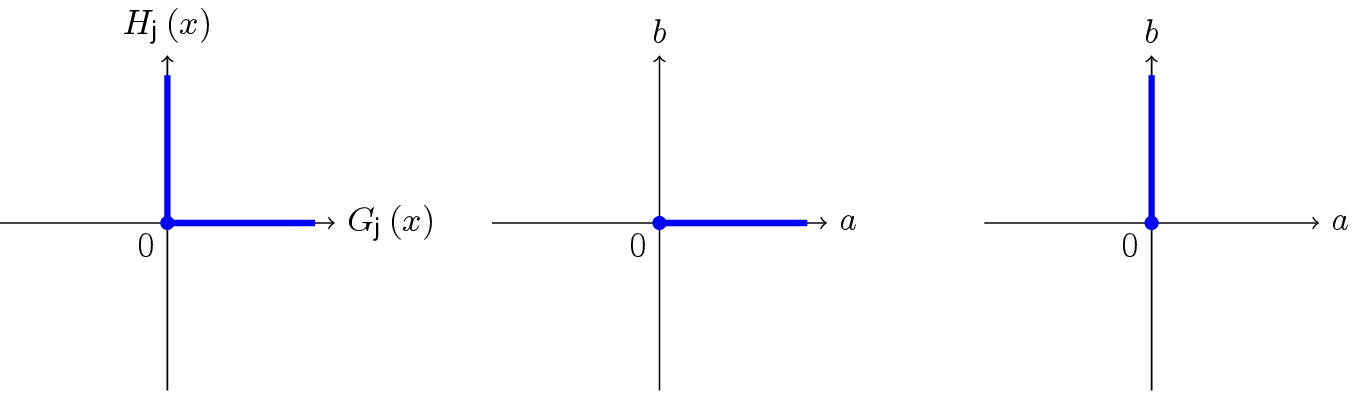}
\caption{Geometric illustrations of $X^\text{CC}$ (left), $S^\text{CC}_1$ (middle), and $S^\text{CC}_2$ (right), respectively.}
\label{fig:MPCC}
\end{figure}
For a feasible point $\bar x\in X^\text{CC}$ of \eqref{eq:MPCC}, let us introduce the following well-known index sets:
\begin{align*}
	I^{+0}(\bar x)&:=\{j\in\{1,\ldots,l\}\,|\,G_j(\bar x)>0,\,H_j(\bar x)=0\},\\
	I^{0+}(\bar x)&:=\{j\in\{1,\ldots,l\}\,|\,G_j(\bar x)=0,\,H_j(\bar x)>0\},\\
	I^{00}(\bar x)&:=\{j\in\{1,\ldots,l\}\,|\,G_j(\bar x)=0,\,H_j(\bar x)=0\}.
\end{align*}
We exploit the calculus rules for the tangent and Fr\'{e}chet normal cone to Cartesian products of (convex) sets, 
see \cite[Proposition~6.41]{RockafellarWets1998},
in order to obtain
\begin{align*}
	\mathcal T_{D^\text{CC}_\alpha}(F(\bar x))&=\mathcal T_{\R^p_-}(g(\bar x))\times\mathcal T_{\{0^q\}}(h(\bar x))
		\times\prod\nolimits_{j=1}^l\mathcal T_{S_{\alpha_j}^\text{CC}}((G_j(\bar x),H_j(\bar x))^\top),\\
	\widehat{\mathcal N}_{D^\text{CC}_\alpha}(F(\bar x))&=\widehat{\mathcal N}_{\R^p_-}(g(\bar x))\times\widehat{\mathcal N}_{\{0^q\}}(h(\bar x))
		\times\prod\nolimits_{j=1}^l\widehat{\mathcal N}_{S_{\alpha_j}^\text{CC}}((G_j(\bar x),H_j(\bar x))^\top)\\
\end{align*}
for each $\alpha\in\mathcal J$.
The tangent and Fr\'{e}chet normal cones to the sets $\R^p_-$ and $\{0^q\}$ 
have been characterized in \cref{ex:standard_nonlinear_programming} already.
A straightforward calculation shows
\begin{align*}
	\mathcal T_{S_{1}^\text{CC}}((G_j(\bar x),H_j(\bar x))^\top)&=
	\begin{cases}
		\R\times\{0\}	& j\in I^{+0}(\bar x),\\
		\varnothing		& j\in I^{0+}(\bar x),\\
		\R_+\times\{0\}	& j\in I^{00}(\bar x),
	\end{cases}	\\
	\mathcal T_{S^\text{CC}_2}((G_j(\bar x),H_j(\bar x))^\top)&=
	\begin{cases}
		\varnothing		& j\in I^{+0}(\bar x),\\
		\{0\}\times\R	& j\in I^{0+}(\bar x),\\
		\{0\}\times\R_+	& j\in I^{00}(\bar x),
	\end{cases}\\
	\widehat{\mathcal N}_{S^\text{CC}_1}((G_j(\bar x),H_j(\bar x))^\top)&=
	\begin{cases}
		\{0\}\times\R	& j\in I^{+0}(\bar x),\\
		\varnothing		& j\in I^{0+}(\bar x),\\
		\R_-\times\R	& j\in I^{00}(\bar x),
	\end{cases}	\\
	\widehat{\mathcal N}_{S^\text{CC}_2}((G_j(\bar x),H_j(\bar x))^\top)&=
	\begin{cases}
		\varnothing		& j\in I^{+0}(\bar x),\\
		\R\times\{0\}		& j\in I^{0+}(\bar x),\\
		\R\times\R_-	& j\in I^{00}(\bar x).
	\end{cases}
\end{align*}
Using $I(\bar x):=\{\alpha\in\mathcal J\,|\,F(\bar x)\in D^\text{CC}_\alpha\}$, we obtain the characterization
\[
	\alpha\in I(\bar x)\,\Longleftrightarrow\,
		\forall j\in I^{+0}(\bar x)\colon\,\alpha_j=1\,\land\,
		\forall j\in I^{0+}(\bar x)\colon\,\alpha_j=2
\]
for arbitrary $\alpha\in\mathcal J$. Thus, MPDC-LICQ from \cref{def:MPDC-LICQ} takes the following form
for problem \eqref{eq:MPCC} at the reference point $\bar x$:
\[
	\left.
		\begin{aligned}
			&0^n=\nabla g(\bar x)^\top\lambda+\nabla h(\bar x)^\top\rho+\nabla G(\bar x)^\top\mu+\nabla H(\bar x)^\top\nu,\\
			&\forall j\notin I^g(\bar x)\colon\,\lambda_j=0,\\
			&\forall j\in I^{+0}(\bar x)\colon\,\mu_j=0,\\
			&\forall j\in I^{0+}(\bar x)\colon\,\nu_j=0
		\end{aligned}
	\right\}
	\,\Longrightarrow\,
	\left\{
		\begin{aligned}
			&\lambda=0^p,\,\rho=0^q,\\
			&\mu=\nu=0^l.
		\end{aligned}
	\right.
\]
Here, the appearing index set $I^g(\bar x)$ has been defined in \cref{ex:standard_nonlinear_programming}.
The above condition is equivalent to the linear independence of the vectors from
\begin{align*}
	&\{\nabla g_j(\bar x)\,|\,j\in I^g(\bar x)\}
	\cup\{\nabla h_j(\bar x)\,|\,j\in\{1,\ldots,q\}\}\\
	&\qquad \cup\{\nabla G_j(\bar x)\,|\,j\in I^{0+}(\bar x)\cup I^{00}(\bar x)\}
	\cup\{\nabla H_j(\bar x)\,|\,j\in I^{+0}(\bar x)\cup I^{00}(\bar x)\}.
\end{align*}
This, however, is precisely the definition of the  prominent constraint qualification MPCC-LICQ, 
see e.g.\ \cite[Definition~2.8]{Ye2005}. A similar observation has been made in 
\cite[Section~4]{Gfrerer2014} using different arguments.
The associated S-stationarity system from \cref{def:SSt} reads as
\begin{align*}
	&0^n=\nabla f(\bar x)+\nabla g(\bar x)^\top\lambda+\nabla h(\bar x)^\top\rho+\nabla G(\bar x)^\top\mu+\nabla H(\bar x)^\top\nu,\\
	&\lambda\geq 0^p,\,\forall j\notin I^g(\bar x)\colon\,\lambda_j=0,\\
	&\forall j\in I^{+0}(\bar x)\colon\,\mu_j=0,\\
	&\forall j\in I^{0+}(\bar x)\colon\,\nu_j=0,\\
	&\forall j\in I^{00}(\bar x)\colon\,\mu_j,\nu_j\leq 0
\end{align*}
and equals the MPCC-tailored system of strong stationarity, see \cite[Definition~2.7]{Ye2005}.
Using the above formulas for the appearing tangent cones, the linearization cone from \eqref{eq:linearization_cone}
is given by
\[
	\mathcal L_{X^\text{CC}}(\bar x)
	=
	\left\{
		d\in\R^n\,\middle|\,
			\begin{aligned}
				\nabla g_j(\bar x)\cdot d&\,\leq\,0\quad j\in I^g(\bar x)\\
				\nabla h_j(\bar x)\cdot d&\,=\,0\quad j\in\{1,\ldots,q\}\\
				\nabla G_j(\bar x)\cdot d&\,=\,0\quad j\in I^{0+}(\bar x)\\
				\nabla H_j(\bar x)\cdot d&\,=\,0\quad j\in I^{+0}(\bar x)\\
				0\leq \nabla G_j(\bar x)\cdot d\,\perp\,\nabla H(\bar x)\cdot d&\,\geq\,0\quad j\in I^{00}(\bar x)
			\end{aligned}
	\right\}
\]
while the critical cone from \eqref{eq:critical_cone} can be easily represented by means of \cref{lem:critical_cone_and_SSt}
whenever the reference point $\bar x$ is S-stationary for \eqref{eq:MPCC}.
As a consequence, \cref{thm:SSt_via_MPDC-LICQ,thm:SONC,thm:SOSC} recover results from the classical
paper \cite{ScheelScholtes2000} while the stability result from \cref{thm:stability_of_S_stationary_points} can be found
in slightly stronger form in \cite[Theorem~4.1]{GuoLinYe2013}.

\subsection{Application to MPVCs}

An optimization problem of type
\begin{equation}\label{eq:MPVC}\tag{MPVC}
	\begin{aligned}
		f(x)&\,\rightarrow\,\min &&&\\
		g_j(x)&\,\leq\,0&\quad&j=1,\ldots,p&\\
		h_j(x)&\,=\,0&&j=1,\ldots,q&\\
		H_j(x)&\,\geq\,0&&j=1,\ldots,l&\\
		G_j(x)H_j(x)&\,\leq\,0&&j=1,\ldots,l
	\end{aligned}
\end{equation}
is called a \emph{mathematical program with vanishing constraints}. The term \emph{vanishing} reflects the observation that whenever
a point $x\in\R^n$ satisfies $H_j(x)=0$ for some $j\in\{1,\ldots,l\}$, then the constraint $G_j(x)H_j(x)\leq 0$ is trivially
satisfied. Problems of type \eqref{eq:MPVC} arise when searching for the optimal design of a truss structure or in the
context of mixed-integer optimal control, see \cite{AchtzigerKanzow2008,Kirches2011,PalagachevGerdts2015}. Theoretical and
numerical results on problems of type \eqref{eq:MPVC} can be found in e.g.\ 
\cite{AchtzigerKanzow2008,AchtzigerKanzowHoheisel2012,Hoheisel2009,HoheiselKanzow2007,HoheiselKanzowSchwartz2012,IzmailovSolodov2009}.

Again, we want to transfer \eqref{eq:MPVC} into a problem of type \eqref{eq:MPDC}. 
Therefore, we define $S^\text{VC}_1,S^\text{VC}_2\subset\R^2$ by means of 
$S^\text{VC}_1:=\R_+\times\{0\}$ and $S^\text{VC}_2:=\R_-\times\R_+$. Furthermore, we set $\mathcal J:=\{1,2\}^l$, 
\[
	\forall \alpha\in\mathcal J\colon\quad
	D_\alpha^\text{VC}:=\R^p_-\times\{0^q\}\times\prod\nolimits_{j=1}^lS_{\alpha_j}^\text{VC},
\]
as well as $D^\text{VC}:=\bigcup_{\alpha\in\mathcal J}D^\text{VC}_\alpha$. 
Using the function $F$ defined in \eqref{eq:definition_constraint_map}, the feasible set of \eqref{eq:MPVC} can be expressed 
in the compact form $X^\text{VC}:=\{x\in\R^n\,|\,F(x)\in D^\text{VC}\}$, see \cref{fig:MPVC}.
\begin{figure}[h]\centering
\includegraphics{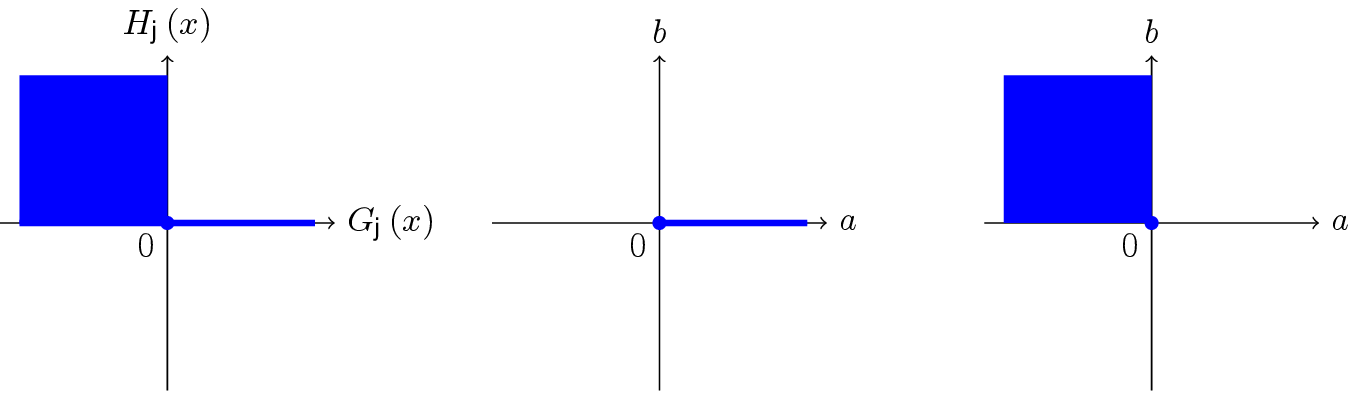}
\caption{Geometric illustrations of $X^\text{VC}$ (left), $S^\text{VC}_1$ (middle), and $S^\text{VC}_2$ (right), respectively.}
\label{fig:MPVC}
\end{figure}
Let us fix a feasible point $\bar x\in X^\text{VC}$ of \eqref{eq:MPVC}. We will exploit the index sets defined below:
\begin{align*}
	I_{+0}(\bar x):=\{j\in\{1,\ldots,l\}\,|\,H_j(\bar x)>0,\,G_j(\bar x)=0\},\\
	I_{+-}(\bar x):=\{j\in\{1,\ldots,l\}\,|\,H_j(\bar x)>0,\,G_j(\bar x)<0\},\\
	I_{0+}(\bar x):=\{j\in\{1,\ldots,l\}\,|\,H_j(\bar x)=0,\,G_j(\bar x)>0\},\\
	I_{0-}(\bar x):=\{j\in\{1,\ldots,l\}\,|\,H_j(\bar x)=0,\,G_j(\bar x)<0\},\\
	I_{00}(\bar x):=\{j\in\{1,\ldots,l\}\,|\,H_j(\bar x)=0,\,G_j(\bar x)=0\}.
\end{align*}
Furthermore, we set $I(\bar x):=\{\alpha\in\mathcal J\,|\,F(\bar x)\in D^{\text{VC}}_\alpha\}$.
Then, we obtain the following characterization for any $\alpha\in\mathcal J$:
\[
	\alpha\in I(\bar x)\,\Longleftrightarrow\,
		\forall j\in I_{0+}(\bar x)\colon\,\alpha_j=1\,\land\,
		\forall j\in I_{+0}(\bar x)\cup I_{+-}(\bar x)\cup I_{0-}(\bar x)\colon\,\alpha_j=2.
\]
Similar as in \cref{sec:MPCC}, the tangent and Fr\'{e}chet normal cones to $S^\textup{VC}_1$
and $S^\textup{VC}_2$ can be computed. As a result, one obtains that 
the constraint qualification  MPDC-LICQ takes the following form for \eqref{eq:MPVC}:
\[
	\left.
		\begin{aligned}
			&0^n=\nabla g(\bar x)^\top\lambda+\nabla h(\bar x)^\top\rho+\nabla G(\bar x)^\top\mu+\nabla H(\bar x)^\top\nu,\\
			&\forall j\notin I^g(\bar x)\colon\,\lambda_j=0,\\
			&\forall j\in I_{+-}(\bar x)\cup I_{0+}(\bar x)\cup I_{0-}(\bar x)\colon\,\mu_j=0,\\
			&\forall j\in I_{+0}(\bar x)\cup I_{+-}(\bar x)\colon\,\nu_j=0
		\end{aligned}
	\right\}
	\,\Longrightarrow
	\,
	\left\{
		\begin{aligned}
			&\lambda=0^p,\,\rho=0^q,\\
			&\mu=\nu=0^l.
		\end{aligned}
	\right.
\]
This condition is equivalent to the linear independence of the vectors from
\begin{align*}
	&\{\nabla g_j(\bar x)\,|\,j\in I^g(\bar x)\}
	\cup\{\nabla h_j(\bar x)\,|\,j\in\{1,\ldots,q\}\}\\
	&\qquad \cup\{\nabla G_j(\bar x)\,|\,j\in I_{+0}(\bar x)\cup I_{00}(\bar x)\}
	\cup\{\nabla H_j(\bar x)\,|\,j\in I_{0+}(\bar x)\cup I_{0-}(\bar x)\cup I_{00}(\bar x)\}
\end{align*}
which is referred to as MPVC-LICQ  in the literature, see \cite[Definition~4.1]{HoheiselKanzow2007}.
The associated S-stationarity system from \cref{def:SSt} reads as follows:
\begin{align*}
	&0^n=\nabla f(\bar x)+\nabla g(\bar x)^\top\lambda+\nabla h(\bar x)^\top\rho+\nabla G(\bar x)^\top\mu+\nabla H(\bar x)^\top\nu,\\
	&\lambda\geq 0^p,\,\forall j\notin I^g(\bar x)\colon\,\lambda_j=0,\\
	&\forall j\in I_{+0}(\bar x)\colon\,\mu_j\geq 0,\\
	&\forall j\in I_{+-}(\bar x)\cup I_{0+}(\bar x)\cup I_{0-}(\bar x)\cup I_{00}(\bar x)\colon\,\mu_j=0,\\
	&\forall j\in I_{+0}(\bar x)\cup I_{+-}(\bar x)\colon\,\nu_j=0,\\
	&\forall j\in I_{0-}(\bar x)\cup I_{00}(\bar x)\colon\,\nu_j\leq 0.
\end{align*}
We note that this system precisely coincides with the system of strong stationarity for \eqref{eq:MPVC}
which has been stated in \cite[Definition~2.1]{HoheiselKanzow2007}.
One can easily check that the linearization cone from \eqref{eq:linearization_cone}
and the critical cone from \eqref{eq:critical_cone} equal the respective cones from 
\cite[Section~4]{HoheiselKanzow2007}. Thus, our \cref{thm:SSt_via_MPDC-LICQ,thm:SONC,thm:SOSC}
precisely recover \cite[Corollary~4.5]{HoheiselKanzow2009} and \cite[Theorems~4.3, 4.4]{HoheiselKanzow2007}.
Additionally, the stability
result from \cref{thm:stability_of_S_stationary_points} is valid for \eqref{eq:MPVC} as well.
To the best of our knowledge, this fact cannot be found in the available literature on \eqref{eq:MPVC}.

\subsection{Application to CCMPs}

Let $\norm{\cdot}{0}\colon\R^n\to\R$ be the map which assigns to each vector from $\R^n$ the number of
its nonzero components. For some constant $\kappa\in\{1,\ldots,n-1\}$, 
\begin{equation}\label{eq:CCMP}\tag{CCMP}
	\begin{aligned}
		f(x)&\,\rightarrow\,\min &&&\\
		g_j(x)&\,\leq\,0&\quad&j=1,\ldots,p&\\
		h_j(x)&\,=\,0&&j=1,\ldots,q&\\
		\norm{x}{0}&\,\leq\,\kappa&&&
	\end{aligned}
\end{equation}
is a nonlinear so-called \emph{cardinality-constrained optimization problem}.
Problems of the form \eqref{eq:CCMP} appear frequently in the context of e.g.\
compressed sensing or portfolio optimization. 
Recently, first- and second-order optimality conditions as well as a relaxation-based numerical solution
method for \eqref{eq:CCMP} were investigated in
\cite{BucherSchwartz2018,BurdakovKanzowSchwartz2016,CervinkaKanzowSchwartz2016}.
The considerations in these papers are based on the surrogate problem
\begin{equation}\label{eq:surrogateCCMP}
	\begin{aligned}
		f(x)&\,\rightarrow\,\min &&&\\
		g_j(x)&\,\leq\,0&\quad&j=1,\ldots,p&\\
		h_j(x)&\,=\,0&&j=1,\ldots,q&\\
		\mathtt e\cdot y-(n-\kappa)&\,\geq\, 0&&&\\
		x_iy_i&\,=\,0&&i=1,\ldots,n&\\
		0\,\leq\,y_i&\,\leq\,1&&i=1,\ldots ,n&
	\end{aligned}
\end{equation}
which is closely related to \eqref{eq:CCMP}, see \cite[Section~3]{BurdakovKanzowSchwartz2016} for details.
Above, $\mathtt e\in\R^n$ represents the all-ones vector.
As suggested in \cite{PanXiuFan2017}, it is also possible to tackle \eqref{eq:CCMP} directly by exploiting
a variational analysis approach. Here, we will strike the latter path.

In order to transfer \eqref{eq:CCMP} into a program of type \eqref{eq:MPDC}, let us introduce the index set
$\mathcal J:=\{\alpha\in\{1,2\}^n\,|\,\sum_{i=1}^n\alpha_i=n+\kappa\}$. Now, for each $\alpha\in\{1,2\}^n$,
we introduce $\R^n_\alpha:=\spa\{\mathtt e_i\,|\,\alpha_i=2\}$ where $\mathtt e_i\in\R^n$ denoted the $i$-th
unit vector from $\R^n$. We set
\[
	\forall\alpha\in\mathcal J\colon\quad
	D^\textup C_\alpha:=\R^p_-\times\{0^q\}\times\R^n_\alpha
\]
as well as $D^\textup C:=\bigcup_{\alpha\in\mathcal J}D^\textup C_\alpha$.
Furthermore, let us define $F\colon\R^n\to\R^{p+q+n}$ by means of
\[
	\forall x\in\R^n\colon\quad
	F(x):=
		\begin{bmatrix}
			g(x)^\top & h(x)^\top & x^\top
		\end{bmatrix}^\top.
\]
Now, the feasible set of \eqref{eq:CCMP} can be 
represented by $X^\textup C:=\{x\in\R^n\,|\,F(x)\in D^\textup C\}$.
Let us fix a feasible point $\bar x\in X^\textup C$ of \eqref{eq:CCMP}. 
We will exploit the index sets
\begin{align*}
	I_{\pm}(\bar x):=\{i\in\{1,\ldots,n\}\,|\,x_i\neq 0\},
	\qquad\qquad
	I_0(\bar x):=\{1,\ldots,n\}\setminus I_\pm(\bar x).
\end{align*}
Furthermore, we will make use of 
$I(\bar x):=\{\alpha\in\mathcal J\,|\,F(\bar x)\in D^\textup C_\alpha\}$.
For arbitrary $\alpha\in\mathcal J$, we obtain 
\[
	\alpha\in I(\bar x)\,\Longleftrightarrow\,
		\forall i\in I_{\pm}(\bar x)\colon\,\alpha_i=2.
\]
Clearly, we have 
\[
	\forall\alpha\in I(\bar x)\colon\quad
	\mathcal T_{\R^n_\alpha}(\bar x)=\R^n_\alpha,
	\qquad
	\widehat{\mathcal N}_{\R^n_\alpha}(\bar x)=\R^n_{3\mathtt e-\alpha}.
\]
This can be used to compute the tangent and Fr\'{e}chet normal cone to $D^\textup{C}_\alpha$
for each $\alpha\in I(\bar x)$.
The resulting constraint qualification MPDC-LICQ for \eqref{eq:CCMP} takes the form
\[
	\left.
		\begin{aligned}
			&0^n=\nabla g(\bar x)^\top\lambda+\nabla h(\bar x)^\top\rho+\mu,\\
			&\forall j\notin I^g(\bar x)\colon\,\lambda_j=0,\\
			&\forall i\in I_{\pm}(\bar x)\colon\,\mu_i=0
		\end{aligned}
	\right\}\,
	\Longrightarrow\,
			\,\lambda=0^p,\,\rho=0^q,\,\mu=0^n\\	
\]
which is equivalent to the linear independence of the vectors from
\[
	\{\nabla g_j(\bar x)\,|\,j\in I^g(\bar x)\}
	\cup
	\{\nabla h_j(\bar x)\,|\,j\in\{1,\ldots,q\}\}
	\cup
	\{\mathtt e_i\,|\,i\in I_0(\bar x)\},
\]
and the latter is well known as CC-LICQ in the literature,
see e.g.\ \cite[Definition~3.11]{CervinkaKanzowSchwartz2016}.
The associated system of S-stationarity from \cref{def:SSt} reads as follows:
\[
	\begin{aligned}
		&0^n=\nabla f(\bar x)+\nabla g(\bar x)^\top\lambda+\nabla h(\bar x)^\top\rho+\mu,\\
		&\lambda\geq 0^p,\,\forall j\notin I^g(\bar x)\colon\,\lambda_j=0,\\
		&\norm{\bar x}{0}=\kappa\,\Longrightarrow\,\forall i\in I_{\pm}(\bar x)\colon\,\mu_i=0,\\
		&\norm{\bar x}{0}<\kappa\,\Longrightarrow\,\mu=0^n.
	\end{aligned}
\]
Defining $\bar y\in\R^n$ by
\[
	\forall i\in\{1,\ldots,n\}\colon\quad
	\bar y_i:=
		\begin{cases}
			0	&i\in I_\pm(\bar x),\\	1&i\in I_0(\bar x),
		\end{cases}
\]
we can check that whenever $\bar x$ is  S-stationary for \eqref{eq:CCMP} in the above sense,
then $(\bar x,\bar y)$ is a feasible point of \eqref{eq:surrogateCCMP} which is
strongly stationary in the sense of \cite[Definition~4.6]{BurdakovKanzowSchwartz2016} whenever
$\norm{\bar x}{0}=\kappa$ holds.
In case $\norm{\bar x}{0}<\kappa$, the S-stationarity conditions from above are more
restrictive than the strong stationarity conditions for \eqref{eq:surrogateCCMP} known from the literature.
However, \cref{thm:SSt_via_MPDC-LICQ} precisely recovers \cite[Proposition~2.1]{BucherSchwartz2018}
in the setting at hand.

Some calculations show that the linearization cone from \eqref{eq:linearization_cone} 
associated with \eqref{eq:CCMP} is given by
\[
	\mathcal L_{X^\textup{C}}(\bar x)
	=
	\left\{
		d\in\R^n\,\middle|\,
			\begin{aligned}
				\nabla g_j(\bar x)\cdot d&\,\leq\,0&&j\in I^g(\bar x)\\
				\nabla h_j(\bar x)\cdot d&\,=\,0&&j\in\{1,\ldots,q\}\\
				|\{i\in I_0(\bar x)\,|\,d_i=0\}|&\,\geq\,n-\kappa&& 
			\end{aligned}
	\right\}.
\]
Whenever $\bar x$ is an S-stationary point of \eqref{eq:CCMP}, then due to
\cref{lem:critical_cone_and_SSt}, for each corresponding multiplier
$(\lambda,\rho,\mu)\in\R^p\times\R^q\times\R^n$ which solves the system of S-stationarity, 
the associated critical
cone from \eqref{eq:critical_cone} is given by
\[
	\mathcal C_{X^\textup{C}}(\bar x)
	=
	\left\{
		d\in\R^n\,\middle|\,
			\begin{aligned}
				\nabla g_j(\bar x)\cdot d&\,\leq\,0&&j\in I^g(\bar x),\,\lambda_j=0\\
				\nabla g_j(\bar x)\cdot d&\,=\,0&&j\in I^g(\bar x),\,\lambda_j>0\\
				\nabla h_j(\bar x)\cdot d&\,=\,0&&j\in\{1,\ldots,q\}\\
				|\{i\in I_0(\bar x)\,|\,d_i=0\}|&\,\geq\,n-\kappa&& 
			\end{aligned}
	\right\}.
\]
Thus, in the context of \eqref{eq:CCMP},
\cref{thm:SONC} precisely recovers \cite[Corollary~3.1]{BucherSchwartz2018} and
\cite[Theorem~4.1]{PanXiuFan2017} while the statement of \cref{thm:SOSC} parallels
\cite[Corollary~3.2]{BucherSchwartz2018} and enhances \cite[Theorem~4.2]{PanXiuFan2017}.
The stability result from \cref{thm:stability_of_S_stationary_points} can be found
in slightly enhanced form in \cite[Corollary~3.3]{BucherSchwartz2018}.

\subsection{Application to MPSCs}

Let us consider so-called \emph{mathematical programs with switching constraints} 
which are optimization problems of the form
\begin{equation}\label{eq:MPSC}\tag{MPSC}
	\begin{aligned}
		f(x)&\,\rightarrow\,\min &&&\\
		g_j(x)&\,\leq\,0&\quad&j=1,\ldots,p&\\
		h_j(x)&\,=\,0&&j=1,\ldots,q&\\
		G_j(x)\cdot H_j(x)&\,=\,0&&j=1,\ldots,l.&
	\end{aligned}
\end{equation}
Models of type \eqref{eq:MPSC} arise from the discretization of so-called
switching-constrained optimal control problems, see e.g.\ \cite{ClasonRundKunisch2017}
and the references therein, as well as the reformulation of logical or-constraints, see
\cite[Section~7]{Mehlitz2018}, or semi-continuity conditions on variables, see 
\cite[Section~5.2.3]{KanzowMehlitzSteck2018}. 
First-order necessary optimality conditions as well as numerical relaxation methods for
problems of type \eqref{eq:MPSC} can be found in \cite{KanzowMehlitzSteck2018,Mehlitz2018}.

Let us transfer \eqref{eq:MPSC} into a program of type \eqref{eq:MPDC}. 
Therefore, we introduce $S^\textup{SC}_1:=\R\times\{0\}$ and $S^\textup{SC}_2:=\{0\}\times\R$
as well as $\mathcal J:=\{1,2\}^l$. We set
\[
	\forall \alpha\in\mathcal J\colon\quad
	D^\textup{SC}_\alpha:=\R^p_-\times\{0^q\}\times\prod\nolimits_{j=1}^lS^\textup{SC}_{\alpha_j}
\]
as well as $D^\textup{SC}:=\bigcup_{\alpha\in\mathcal J}D^\textup{SC}_\alpha$.
Using the mapping $F$ defined in \eqref{eq:definition_constraint_map}, the feasible set of
\eqref{eq:MPSC} can be represented by $X^\textup{SC}:=\{x\in\R^n\,|\,F(x)\in D^\textup{SC}\}$.
The variational geometry of $X^\textup{SC}$ is visualized in \cref{fig:MPSC}.
\begin{figure}[h]\centering
\includegraphics[width=0.95\textwidth]{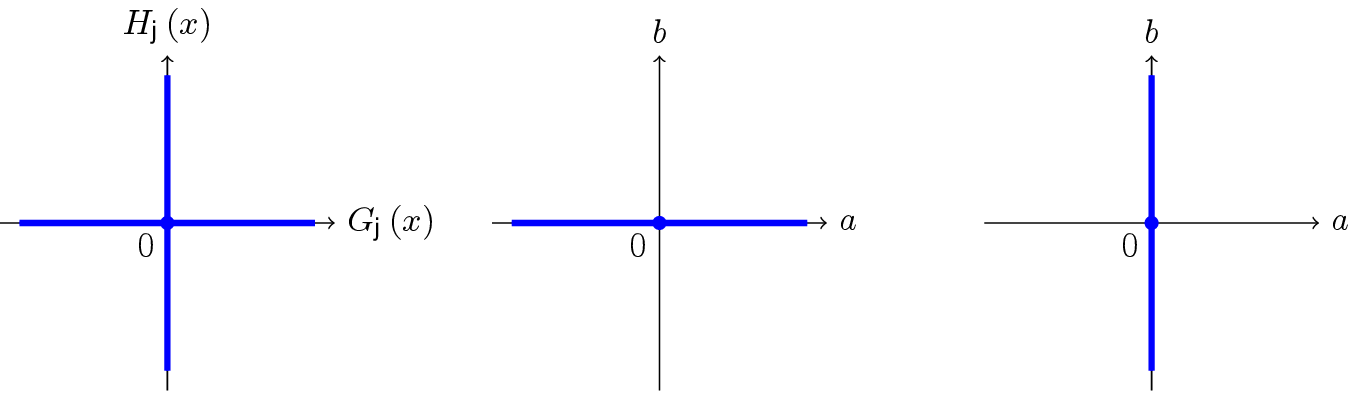}
\caption{Geometric illustrations of $X^\textup{SC}$ (left) , $S^\textup{SC}_1$ (middle), and $S^\textup{SC}_2$ (right), respectively.}
\label{fig:MPSC}
\end{figure}
For a feasible point $\bar x\in X^\textup{SC}$ of \eqref{eq:MPSC}, we introduce the following index sets:
\begin{align*}
	I^{G}(\bar x)&:=\{j\in\{1,\ldots,l\}\,|\,G_j(\bar x)=0,\,H_j(\bar x)\neq 0\},\\
	I^{H}(\bar x)&:=\{j\in\{1,\ldots,l\}\,|\,G_j(\bar x)\neq 0,\,H_j(\bar x)=0\},\\
	I^{GH}(\bar x)&:=\{j\in\{1,\ldots,l\}\,|\,G_j(\bar x)=0,\,H_j(\bar x)=0\}.
\end{align*}
Furthermore, we set $I(\bar x):=\{\alpha\in\mathcal J\,|\,F(\bar x)\in D^\textup{SC}_\alpha\}$.
Thus, we have
\[
	\alpha\in I(\bar x)
	\,\Longleftrightarrow\,
	\forall j\in I^G(\bar x)\colon\,\alpha_j=2\,\land\,\forall j\in I^H(\bar x)\colon\,\alpha_j=1
\]
for each $\alpha\in\mathcal J$.
Computing the tangent and Fr\'{e}chet normal cones to $S^\textup{SC}_1$ and $S^\textup{SC}_2$, 
we obtain that the constraint qualification MPDC-LICQ takes the following form for \eqref{eq:MPSC}:
\[
	\left.
		\begin{aligned}
			&0^n=\nabla g(\bar x)^\top\lambda+\nabla h(\bar x)^\top\rho+\nabla G(\bar x)^\top\mu+\nabla H(\bar x)^\top\nu,\\
			&\forall j\notin I^g(\bar x)\colon\,\lambda_j=0,\\
			&\forall j\in I^H(\bar x)\colon\,\mu_j=0,\\
			&\forall j\in I^G(\bar x)\colon\,\nu_j=0
		\end{aligned}
	\right\}
	\,\Longrightarrow\,
	\left\{
		\begin{aligned}
			&\lambda=0^p,\,\rho=0^q,\\
			&\mu=\nu=0^l.
		\end{aligned}
	\right.	
\]
We note that this is equivalent to the linear independence of all the vectors from
\begin{align*}
	&\{\nabla g_j(\bar x)\,|\,j\in I^g(\bar x)\}
	\cup\{\nabla h_j(\bar x)\,|\,j\in\{1,\ldots,q\}\}\\
	&\qquad \cup\{\nabla G_j(\bar x)\,|\,j\in I^G(\bar x)\cup I^{GH}(\bar x)\}
	\cup\{\nabla H_j(\bar x)\,|\,j\in I^H(\bar x)\cup I^{GH}(\bar x)\},
\end{align*}
and this condition is called MPSC-LICQ in the literature, see \cite[Definition~4.4]{Mehlitz2018}.
The associated S-stationarity system from \cref{def:SSt} is given by
\begin{align*}
	&0^n=\nabla f(\bar x)+\nabla g(\bar x)^\top\lambda+\nabla h(\bar x)^\top\rho
		+\nabla G(\bar x)^\top\mu+\nabla H(\bar x)^\top \nu,\\
	&\lambda\geq 0^p,\,\forall j\notin I^g(\bar x)\colon\,\lambda_j=0,\\
	&\forall j\in I^H(\bar x)\cup I^{GH}(\bar x)\colon\,\mu_j=0,\\
	&\forall j\in I^G(\bar x)\cup I^{GH}(\bar x)\colon\,\nu_j=0
\end{align*}
and equals the problem-tailored system of strong stationarity as it has been stated in
\cite[Definition~4.3]{Mehlitz2018}. The results of \cref{thm:SSt_via_MPDC-LICQ} can
be found in \cite[Theorem~4.5]{Mehlitz2018}.

One can easily check that the linearization cone from \eqref{eq:linearization_cone}
possesses the form
\[
	\mathcal L_{X^\textup{SC}}(\bar x)=
		\left\{
			d\in\R^n\,\middle|\,
				\begin{aligned}
					\nabla g(\bar x)\cdot d&\,\leq\,0&&j\in I^g(\bar x)\\
					\nabla h(\bar x)\cdot d&\,=\,0&& j\in\{1,\ldots,q\}\\
					\nabla G(\bar x)\cdot d&\,=\,0&& j\in I^G(\bar x)\\
					\nabla H(\bar x)\cdot d&\,=0\,&& j\in I^H(\bar x)\\
					(\nabla G(\bar x)\cdot d)(\nabla H(\bar x)\cdot d)&\,=\,0&& j\in I^{GH}(\bar x)
				\end{aligned}
		\right\}
\]
while we obtain 
\[
	\mathcal C_{X^\textup{SC}}(\bar x)=
		\left\{
			d\in\R^n\,\middle|\,
				\begin{aligned}
					\nabla g(\bar x)\cdot d&\,\leq\,0&&j\in I^g(\bar x),\,\lambda_j=0\\
					\nabla g(\bar x)\cdot d&\,=\,0&& j\in I^g(\bar x),\,\lambda_j>0\\
					\nabla h(\bar x)\cdot d&\,=\,0&& j\in\{1,\ldots,q\}\\
					\nabla G(\bar x)\cdot d&\,=\,0&& j\in I^G(\bar x)\\
					\nabla H(\bar x)\cdot d&\,=0\,&& j\in I^H(\bar x)\\
					(\nabla G(\bar x)\cdot d)(\nabla H(\bar x)\cdot d)&\,=\,0&& j\in I^{GH}(\bar x)
				\end{aligned}
		\right\}
\]
for the associated critical cone provided $\bar x$ is an S-stationary point of \eqref{eq:MPSC}
with associated multipliers $(\lambda,\rho,\mu,\nu)\in\R^p\times\R^q\times\R^l\times\R^l$, 
see \cref{lem:critical_cone_and_SSt}.
Based on this critical cone, \cref{thm:SONC,thm:SOSC} provide a necessary and sufficient
second-order optimality condition for \eqref{eq:MPSC}. 
Furthermore, \cref{thm:stability_of_S_stationary_points} yields
a criterion which ensures local uniqueness of S-stationary points associated with
\eqref{eq:MPSC}. To the best of our knowledge, these are new results on the problem class
\eqref{eq:MPSC}.

\section{Final remarks}\label{sec:conclusions}

In this paper, we introduced a reasonable abstract version of the prominent linear independence constraint
qualification which applies to mathematical programs with disjunctive constraints.
We were able to derive first- and second-order optimality conditions based on strongly stationary
points under validity of this constraint qualification in elementary way. 
Finally, we applied our findings to several different instances of disjunctive programs in
order to underline that this new constraint qualification is reasonable. 
By means of switching-constrained mathematical problems, it has been demonstrated that our
theory does not only recover well-known results from the literature but can be used to
infer new results on specific instances of disjunctive programming as well.

\subsection*{Acknowledgments}

We would like to thank two anonymous reviewers for several valuable comments and remarks which
led to a significant improvement of this paper during revision. 
Particularly, we are in debt to one of the referees for suggesting the current version of \cref{lem:stability_property}
as well as its proof.

\bibliographystyle{plainnat}
\bibliography{references}

\end{document}